\documentclass[12pt]{amsart}

\usepackage[T1]{fontenc}
\usepackage[utf8]{inputenc}
\usepackage[english]{babel}
\usepackage{amsmath,amsthm,amssymb,enumerate,bbm,hyperref,tikz}
\usepackage[all]{xy}

\makeatletter
\@namedef{subjclassname@2010}{%
  \textup{2010} Mathematics Subject Classification}
\makeatother

\theoremstyle{plain}
\newtheorem{thm}{Theorem}
\newtheorem{ithm}{Theorem}
\newtheorem{lem}{Lemma}[subsection]
\newtheorem{prop}[lem]{Proposition}
\newtheorem{cor}[lem]{Corollary}
\newtheorem{pb}{Problem}

\theoremstyle{definition}
\newtheorem{defn}[lem]{Definition}

\theoremstyle{remark}
\newtheorem{rem}[lem]{Remark}
\newtheorem{exple}[lem]{Example}

\renewcommand{\k}{\mathbbm k}

\begin{document}

\dedicatory{Dedicated to Claude Cibils for his 60th birthday}

\title[Partial Relation Extensions]{Representation Theory of Partial
  Relation Extensions}

\author[I. Assem]{Ibrahim Assem}
\author[J. C. Bustamante]{Juan Carlos Bustamante}

\address[Ibrahim Assem, Juan Carlos Bustamante]{Département de
  Mathématiques, Université de Sherbrooke, Sherbrooke, Québec, Canada
  J1K2R1}
\email{\begin{tabular}[t]{l}
         ibrahim.assem@usherbrooke.ca\\
         juan.carlos.bustamante@usherbrooke.ca
         \end{tabular}}

\curraddr[Juan Carlos Bustamante]{Mathematics Department, Champlain
  College, Lennoxville, 2580 Rue College, Sherbrooke, Québec, Canada
  J1M 2K3}

\author[J. Dionne]{Julie Dionne}

\address[Julie Dionne]{}
\email{dionnej@gmail.com}

\author[P. Le Meur]{Patrick Le Meur}

\address[Patrick Le Meur]{ Laboratoire de Math\'ematiques,
  Universit\'e Blaise Pascal \& CNRS, Complexe Scientifique Les
  Cézeaux, BP 80026, 63171 Aubi\`ere cedex, France}

\curraddr[Patrick Le Meur]{Univ Paris Diderot, Sorbonne Paris Cit\'e, Institut de
  Math\'ematiques de Jussieu-Paris Rive Gauche, UMR 7586, CNRS, Sorbonne
  Universit\'es, UPMC Univ Paris 06, F-75013, Paris, France}

\email{patrick.le-meur@imj-prg.fr}

\author[D. Smith]{David Smith}

\address[David Smith]{Department of Mathematics, Bishop's University. Sherbrooke,
  Québec, Canada, J1M1Z7}

\email{dsmith@ubishops.ca}

\date{}

\begin{abstract}
  Let $C$ be  a finite dimensional algebra of global dimension at most two.
A partial relation extension is any trivial extension of $C$ by
a direct summand of its relation $C-C$-bimodule. When $C$ is a tilted algebra,
this construction provides an intermediate class of algebras between
tilted and cluster tilted algebras. The text investigates the
representation theory of partial relation extensions. When $C$ is tilted, any
complete slice in the Auslander-Reiten quiver of $C$ embeds as a local
slice in the Auslander-Reiten quiver of the partial relation
extension. Moreover, a systematic way 
of producing partial relation extensions is introduced by considering
direct sum decompositions of the potential arising from a minimal
system of relations of $C$. 
\end{abstract}



\maketitle

\section*{Introduction}
\label{sec:introduction}

Cluster tilted algebras were introduced in \cite{MR2247893} and
independently in \cite{CCS06a} for the $\mathbb A$ case, as a
by-product of the now extensive theory of cluster algebras of Fomin and
Zelevinsky. They have been the subject of many investigations. In
particular, it was proved in \cite{MR2409188} that a cluster tilted
algebra can always be written as the relation extension of a tilted
algebra $C$, that is, the trivial extension of $C$ by the so-called
relation bimodule $E={\rm Ext}^2_C(DC,C)$. Tilted algebras have
been characterised by the existence of complete slices in their module
categories, see, for instance, \cite{MR2197389}. It was proven in
\cite{ABS08} that any complete slice in the module category of a
tilted algebra $C$ embeds in the module category of its relation
extension $\widetilde C$ as what is called a local slice. However, as
seen in \cite{ABS08}, the existence of local slices does not
characterise cluster tilted 
algebras, and it was asked there which algebras are characterised by
the existence of local slices. Our objective in the present paper is
to exhibit another natural class of algebras admitting local slices.

Because cluster tilted algebras are Jacobian algebras of quivers with
potential, as shown in \cite{MR2823864}, we take this context as our
starting point. We define the 
notion of direct sum decomposition of the Keller 
potential of the relation extension of a triangular algebra $C$ with
global dimension at most 
two. In this case, a direct sum decomposition of the potential
associated with the relation extension of $C$ induces
a direct sum decomposition of the
relation bimodule. It is reasonable to expect that the converse
statement also holds true. We can prove this converse in two cases
where a minimal system of minimal relations is known, namely the
cluster tilted algebras with a cyclically oriented quiver of
\cite{MR3169475}, which include all the representation-finite cluster
tilted algebras, see \cite{MR2271343}, and the cluster tilted algebras
of type $\widetilde{\mathbb A}$ of \cite{MR2592019}. Referring to
section~\ref{sec:decomp-potent-relat} for the definitions, our first
theorem reads as follows.

\begin{ithm}[Propositions~\ref{sec:from-direct-decomp-2},
  \ref{sec:from-direct-decomp-6} and \ref{sec:from-direct-decomp-7}]
  \label{sec:introduction-1}
  Let $C=\k Q/I$ be a triangular algebra of global dimension at most two, and
  $W$ be the Keller potential of its relation extension associated
  with a minimal system of relations in $I$. If
  $W=W'\oplus W''$ is a direct sum decomposition and $E',E''$ are the
  partial relation bimodules corresponding to $W',W''$ respectively,
  then
  \[
  E=E'\oplus E''
  \]
  as $C-C$-bimodules.

  Conversely, if $\widetilde C=C\ltimes {\rm Ext}^2_C(DC,C)$ is a cluster
  tilted algebra with a cyclically oriented quiver or a cluster tilted
  algebra of type $\widetilde{\mathbb A}$ and $E=E'\oplus E''$ is a
  direct sum decomposition of $E$ as $C-C$-bimodules, then there
  exists a direct sum decomposition of the Keller potential
  \[
  W=W'\oplus W''
  \]
  such that $E',E''$ are the partial relation bimodules
  corresponding to $W',W''$ respectively.  
\end{ithm}

We then define the class of algebras we are interested in. Let $C$ be
a triangular algebra of global dimension at most two, and $E=E'\oplus
E''$ be a direct sum of $C-C$-bimodules, then the algebra $B=C\ltimes
E'$ is called a partial relation extension of $C$. Because it is
easily shown that $\widetilde C=B\ltimes E''$, partial relation
extensions can be thought of as an intermediate class of algebras
between tilted and cluster tilted algebras (or more generally, between
a triangular algebra of global dimension at most two, and its relation
extension). The bound quiver of a
partial relation extension is easily computed and we then proceed to
study its module category, obtaining the following theorem when the
original algebra $C$ is tilted.

\begin{ithm}
  \label{sec:introduction-2}
  Let $H$ be a hereditary algebra, $\mathcal C_H$ its cluster
  category, $T_H$ be a tilting $H$-module and
  $C={\rm End}_H(T)$. Then, there exists an ideal $\mathcal K$ in the
  cluster category such that the composition  $( -\underset{\widetilde
      C}{\otimes} B) \circ {\rm Hom}_{\mathcal C_H}(T,-)\colon
    \mathcal C_H\to {\rm mod}\,\widetilde C\to {\rm mod}\,B$
  induces an equivalence
  \[{\rm mod}\,B\simeq \mathcal C_H/\mathcal K\,.\]
\end{ithm}

The ideal $\mathcal K$ is characterised by approximations in the
cluster category. It is
important to observe that, in contrast to what happens for cluster
tilted algebras, factoring by $\mathcal K$ does not mean 
simply deleting finitely many objects of $\mathcal C_H$: we may have
$H$ representation-infinite and $B$ representation-finite. As an easy
consequence of our theorem~\ref{sec:introduction-2}, we obtain a full
and dense functor from the module category of the cluster repetitive
algebra of $C$ to ${\rm mod}\,B$. Returning to our original
motivation, we finally prove the following result.

\begin{ithm}
  \label{sec:introduction-3}
  Let $C$ be a tilted algebra and $A$ be an algebra such that there
  exist surjective algebra morphisms $\widetilde C\twoheadrightarrow
  A\twoheadrightarrow C$. Then any complete slice in $\Gamma({\rm
    mod}\,C)$ embeds as a local slice in $\Gamma({\rm mod}\,A)$. In
  particular, partial relation extensions admit local slices.
\end{ithm}

Notice however that H. Treffinger \cite{treffinger} has obtained a
very large class of algebras having local slices, comprising partial
relation extensions.

We devote a section of the paper to the proof of each of the stated theorems.

% \section{Notation}
% \label{sec:notation}

\section{Decomposition of the potential and the relation bimodule}
\label{sec:decomp-potent-relat}

\subsection{Decompositions of a potential}
\label{sec:decomp-potent-exampl}

Let $(Q,W)$ be a pair consisting of a finite quiver $Q$ and a
potential $W$, that is, a linear combination of oriented cycles of $Q$. Define a relation 
between the (oriented) cycles which appear as summands of $W$ as follows:
$\gamma\sim\gamma'$ 
whenever there exists an arrow $\alpha\in Q_1$ which is common to both
 $\gamma$ and $\gamma'$. This relation is reflexive and
symmetric, let $\approx$ be its transitive closure (that is, the
smallest equivalence relation containing it).

Two cycles $\gamma$ and $\gamma'$ are called \emph{independent} if
$\gamma\not\approx \gamma'$, and \emph{dependent} if $\gamma\approx
\gamma'$.

A sum decomposition of the potential
\[
W=W'+W''
\]
is said to be \emph{direct} if, whenever $\gamma'$ is any cycle in $W'$ and
$\gamma''$ is any cycle in $W''$, we have $\gamma'\not\approx
\gamma''$. We denote a direct sum decomposition of the potential as
$W=W'\oplus W''$.

\begin{exple}\ 
  \begin{enumerate}[(a)]
  \item Let $(Q,W)$ be the quiver
    
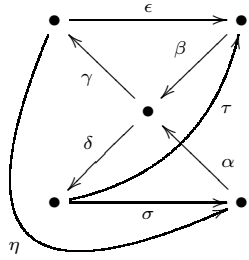
\[
    \xymatrix{
      \bullet \ar[rr]^{\epsilon} \ar@/_15ex/_{\eta}[rrdd]& & \bullet \ar[ld]_{\beta} \\
      & \bullet \ar[lu]^{\gamma}  \ar[ld]_{\delta} \\
      \bullet \ar[rr]_{\sigma} \ar@/_4ex/[rruu]_(.7){\tau}& & \bullet
      \ar[lu]_(.3){\alpha}  
    }
    \]
    with $W=\beta\gamma \epsilon +\beta\delta\tau + \alpha\gamma \eta
    + \alpha \delta \sigma$. Here, the four summands of the potential
    are pairwise dependent.
  \item Let $(Q,W)$ be the quiver
    
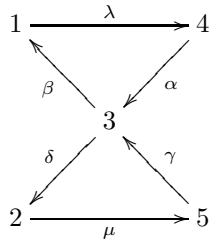
\[
    \xymatrix{
      1 \ar[rr]^\lambda & & 4 \ar[ld]^\alpha \\
      & 3 \ar[lu]^\beta \ar[ld]_\delta \\
      2 \ar[rr]_\mu & & 5 \ar[lu]_\gamma
    }
    \]
    with $W=\alpha\beta\lambda +\gamma\delta\mu$. Here the two
    cycles $\alpha\beta\lambda$ and $\gamma\delta\mu$ are
    independent so the decomposition $W=W_1+W_2$ with
    $W_1=\alpha\beta\lambda$ and $W_2= 
    \gamma\delta\mu$ is direct, and $W=W_1\oplus W_2$.
  \end{enumerate}
\end{exple}

\subsection{Induced decompositions of the relation bimodule}
\label{sec:from-direct-decomp}

Our objective is to apply the notion of direct sum decompositions of the
potential to the study of
cluster tilted algebras.  We refer the reader to \cite{MR2247893} and
to \cite{MR2409188} for general background on cluster tilted
algebras. In particular let $C$ be a triangular algebra of global 
dimension at most two and consider the $C-C$-bimodule 
$E={\rm Ext}^2_{C}(DC,C)$ equipped with the natural left and right
actions of $C$. This bimodule $E$ is called the \emph{relation
  bimodule} and the trivial extension algebra $\widetilde C=C\ltimes
E$ is called the \emph{relation extension} of $C$.
The best known
class of relation extensions is provided by the cluster tilted
algebras: it is shown in \cite[(3.4)]{MR2409188} that, if $C$ is a
tilted algebra,  then $\widetilde C$ is cluster tilted, and every
cluster tilted algebra arises in this way.

The bound quiver of a relation extension is constructed as
follows. Let $C=\k Q/I$ be an admissible presentation of $C$.
A subset
$R=\{\rho_1,\ldots, \rho_t\}$ of $\bigcup_{x,y\in 
  Q_0}e_xIe_y$ is called a  \emph{system of
  relations} for $C$ if  $R$, but no proper subset of $R$, generates $I$ as a two-sided
ideal, see \cite[(1.2)]{MR724715}. The ordinary quiver $\widetilde Q$ of
$\widetilde C$ has the same vertices as those of $Q$, while the set of
arrows in $\widetilde Q$ from $x$ to $y$, say, equals the set of
arrows in $Q$ from $x$ to $y$, plus, for each relation $\rho\in R\cap
e_yIe_x$, a so-called \emph{new} arrow $\alpha_\rho\colon x\to y$, see
\cite[(2.6)]{MR2409188}. Thus $\widetilde C$ is not triangular unless
$C$ is hereditary and, if $R=\{\rho_1,\ldots,\rho_r\}$ is as above,
and the new arrow $\alpha_i$ corresponds to $\rho_i$, then
$\alpha_i\rho_i$ is an oriented cycle in $\widetilde Q$. We define the
\emph{Keller potential} (associated with $R$) by setting
\[
W=\sum_{i=1}^t\alpha_i\rho_i\,.
\]
Oriented cycles in potentials are, as usual, considered up to cyclic
permutations: two potentials are called \emph{cyclically equivalent}
if their difference lies in the linear span of all elements of the
form
$\gamma_1\gamma_2\cdots\gamma_m-\gamma_m\gamma_1\cdots\gamma_{m-1}$
where $\gamma_1\cdots\gamma_m$ is an oriented cycle. For a given arrow
$\beta$, the \emph{cyclic partial derivative} $\partial_\beta$ of
$W$ is defined on each cyclic summand $\gamma_1\cdots\gamma_m$ of $W$
by
\[
\partial_\beta(\gamma_1\cdots\gamma_m)=\sum_{\beta=\gamma_i}\gamma_{i+1}\cdots\gamma_m\gamma_1\cdots
\gamma_{i-1}\,.
\]
In particular, cyclic derivatives are invariant under cyclic
permutations. The \emph{Jacobian algebra} $J(\widetilde Q,W)$ is the one
given by the quiver $\widetilde Q$ bound by all partial cyclic
derivatives $\partial_\beta W$ of the Keller potential $W$ with respect
to each arrow $\beta\in \widetilde Q_1$. Then the relation extension
$\widetilde C$ is isomorphic to $J(\widetilde Q,W)/\mathcal J$ where $\mathcal J$
is the square of the ideal of $J(\widetilde Q,W)$ generated by the new arrows,
see \cite[Lemma 5.2]{AGST}. If, in particular, $C$ is tilted, so that
$\widetilde C$ is cluster tilted, then
$\widetilde C\simeq J(\widetilde Q,W)$, see for instance
\cite{MR2795754}.

Setting $\widetilde C=\k \widetilde Q/\widetilde I$, we recall from
\cite[(2.4)]{MR2409188} that the classes of arrows 
(modulo $\widetilde I$) which belong to $\widetilde Q_1\backslash Q_1$
are the generators of the $C-C$-bimodule $E$.

Before proving the main result of the subsection, we need a technical
lemma. We assume that $C$ is a triangular algebra of global dimension
at most two, and that $\widetilde C$ is its relation extension.
\begin{lem}
  With the above notation, consider a partition of the set of new
  arrows $\widetilde Q_1\backslash Q_1=F'\cup F''$. Let
  $E',E''$ be the subbimodules of $E$ generated by the classes of the
  arrows in $F'$ and $F''$, respectively. If
  $E'\cap E''\neq 0$ then there exist oriented cycles
  $\gamma',\gamma''$ in $W$ such that
  \begin{enumerate}
  \item $\gamma'$ has one or two arrows in $\widetilde Q_1\backslash
    Q_1$, and at least one of them lies in $F'$,
  \item $\gamma''$ has one or two arrows in $\widetilde Q_1\backslash
    Q_1$, and at least one of them lies in $F''$,
  \item $\gamma'$ and $\gamma''$ have a common arrow,
  \item $\gamma'$ has two arrows in $\widetilde Q_1\backslash Q_1$ if
    and only if so does $\gamma''$, in which case $\gamma'$ and
    $\gamma''$ have a common arrow in $\widetilde Q_1\backslash Q_1$.    
  \end{enumerate}
\end{lem}
\begin{proof}
  Let $e$ be a nonzero element in $E'\cap E''$.
  There exist paths $u_1,\ldots,u_m$, $v_1,\ldots,v_n$ and scalars
  $\lambda_1,\ldots,\lambda_m,\mu_1,\ldots, \mu_n$ satisfying
  the following conditions
  \begin{enumerate}[(a)]
  \item $e$ equals both classes of $\sum_i\lambda_i u_i$ and $\sum_j\mu_j v_j$,
  \item each $u_i$ has exactly one arrow from $\widetilde Q_1\backslash
    Q_1$ and that arrow lies in $F'$, we denote this arrow by $\alpha'_i$,
  \item each $v_i$ has exactly one arrow from $\widetilde Q_1\backslash
    Q_1$ and that arrow lies in $F''$, we denote this arrow by $\alpha''_i$.
  \end{enumerate}
  Therefore, there exist paths $a_1,\ldots,a_N,b_1,\ldots,b_N$, 
  scalars $t_1,\ldots,t_N$ and arrows $\beta_1,\ldots,\beta_N$ such
  that
  \[
  \sum_i\lambda_i u_i - \sum_j\mu_j v_j =
  \sum_{\ell}t_{\ell}a_{\ell}\cdot \partial_{\beta_\ell}W\cdot b_{\ell}\,.
  \]
  In view of condition (a) above and because $e\neq 0$, there exists $\ell$ such that the
  expression $a_{\ell}\cdot \partial_{\beta_{\ell}} W\cdot b_{\ell}$
  contains both $u_i$ and $v_j$ for some indices $i,j$.
  Note that neither $\alpha'_i$ nor $\alpha''_j$ appears in some $a_{\ell}$
  or $b_{\ell}$ for, otherwise, both would appear in $u_i$ and $v_j$,
  thus contradicting conditions (b) and (c) above. Hence,  there exist
  oriented cycles $\gamma',\gamma''$ that appear in $W$, that contain
  $\alpha'_i$ and $\alpha''_j$, respectively, and that both contain
  $\beta_{\ell}$.

  Since any cycle in $W$ contains at most one arrow from $\widetilde
  Q_1\backslash Q_1$ it follows that $\gamma'$ contains at most two
  arrows from $\widetilde Q_1\backslash Q_1$ (namely $\alpha_i'\in F'$ and
  possibly $\beta_\ell$). Whence (1). Assertion (2) follows from similar
  considerations. Moreover, $\gamma'$ and $\gamma''$ have the arrow
  $\beta_{\ell}$ in common. This shows (3) and (4).
\end{proof}

In view of the preceding lemma, we define for each direct summand $W'$
of
the potential $W$  in
$\k \widetilde Q$ the subbimodule $E'$ of $E$ as follows: $E'$ is
generated by the classes of arrows in $\widetilde Q_1\backslash Q_1$
appearing in a cycle of $W'$. We call $E'$ the \emph{partial relation
  bimodule} corresponding to $W'$.

\begin{prop}
  \label{sec:from-direct-decomp-2}
  Let $W=W'\oplus W''$ be a direct sum decomposition of the potential. Then
  $E=E'\oplus E''$ where $E'$ and $E''$ are the partial relation
  bimodules corresponding to $W'$ and $W''$, respectively.
\end{prop}
\begin{proof}
  Let $F'$ and $F''$ be the set of arrows in $\widetilde Q_1\backslash
  Q_1$ appearing in a cycle from $W'$ and $W''$, respectively. By
  construction of $W$, the union $F'\cup F''$ equals
  $\widetilde Q_1\backslash Q_1$. And because the decomposition
  $W=W'+W''$ is direct, $F'\cap F''=\emptyset$. The preceding lemma
  therefore applies: because $W=W'+W''$ is a direct sum decomposition,
  it entails that $E'\cap E''=0$. On the other hand $E=E'+E''$ because $F'\cup
  F''=\widetilde Q_1\backslash Q_1$.
\end{proof}

It is natural to ask if, conversely, given a direct sum decomposition
of the relation bimodule $E=E'\oplus E''$, one can get a corresponding
decomposition of the potential. The next two subsections are devoted
to this problem.

In order that the converse process be possible, it seems to be needed that a
presentation of the cluster tilted algebra by minimal relations be
given by the potential. It is known that this is not always the case,
see \cite[Example 4.3]{MR3169475}. Recall that, following \cite{MR2271343}, a
\emph{minimal relation} in a bound quiver $(Q,I)$ is any element of
$I$ not lying in $\underline r I+I \underline r$, where $\underline r$
denotes the two-sided ideal of $\k Q$ generated by all the arrows of
$Q$. The problem of finding systems of minimal relations for
a cluster tilted algebra or, more generally, Jacobian algebras of
quivers with potentials, is a basic one. It was first solved for
representation-finite cluster tilted algebras in \cite{MR2271343}, then for
the cluster tilted algebras having a cyclically oriented quiver
\cite{MR3169475}. The latter class includes the representation-finite cluster
tilted 
algebras. Also it was solved for Jacobian algebras arising from
surfaces without punctures and in particular for cluster tilted algebras of type
$\widetilde{\mathbb A}$ in \cite{MR2592019}. We are not aware of other cases where
the solution is known.
We pose the following problem.
\begin{pb}
  Given a system of minimal relations on a Jacobian algebra, which
  conditions are necessary on this system in order for the converse of 
  Proposition~\ref{sec:from-direct-decomp-2} be valid?
\end{pb}

\subsection{Induced decompositions of the potential: the cyclically oriented case}
\label{sec:from-direct-decomp-1}

Here we prove this converse in the two particular cases where systems of
minimal relations are known. We start with
algebras having cyclically oriented quivers.
We recall from \cite{MR3169475} that a quiver is
called \emph{cyclically oriented} if each chordless cycle is an
oriented cycle. Here is a summary
 of the combinatorial
  properties of $\widetilde Q$ that follow from the fact that it is
  cyclically oriented (see \cite[Proposition 1.1, Proposition
  3.5]{MR3169475}).
  \begin{enumerate}[(a)]
  \item Let $a\in \widetilde Q_1$ lie in an oriented cycle. Then the
    sum of all the paths antiparallel to $a$ is a minimal relation.
  \item Any minimal relation is proportional to one as above.
  \item Let $a\in \widetilde Q_1$ lie in an oriented cycle. Then $a$
    has no parallel arrow and two distinct paths antiparallel to $a$
    have no common vertex but their source and target.
  \end{enumerate}
Here, two oriented paths, say from $x$ to $y$ and from $x'$ to $y'$,
respectively, are called \emph{parallel} whenever $x=x'$ and $y=y'$,
and they are called \emph{antiparallel} whenever $x=y'$ and $y=x'$.
\begin{prop}
  \label{sec:from-direct-decomp-3}
  Let $\widetilde C$ be a cluster tilted algebra with a cyclically
  oriented quiver. Assume $E=E'\oplus E''$ is a nontrivial direct sum
  decomposition of $E$ as $C-C$-bimodule. Then there exists a
  nontrivial direct sum
  decomposition $W=W'+W''$ of the Keller potential such that $E'$, $E''$ are
  respectively the partial relation bimodules corresponding to $W', W''$.
\end{prop}
\begin{proof}
  The direct sum decomposition of $C-C$-bimodules  $E=E'\oplus E''$ induces a
  decomposition ${\rm top}\,E={\rm top}\,E'\oplus {\rm top}\,E''$. Let 
  $\Sigma$ be the set of couples $(x,y)$ of vertices such that ${\rm
    Ext}^2_C(S_x,S_y)\neq 0$. Recall that ${\rm dim}_{\k}\,{\rm
    Ext}^2_C(S_x,S_y)\leqslant 1$ for any couple $(x,y)$. Hence there
  exists a nontrivial partition $\Sigma=\Sigma'\cup \Sigma''$ such that
  $e_y{\rm top}(E')\,e_x=e_y{\rm top}(E)x_x$ if $(x,y)\in \Sigma'$ and
  $e_y{\rm top}(E'')\,e_x=e_y{\rm top}(E)x_x$ if $(x,y)\in
  \Sigma''$.
  Since $\widetilde Q$ is cyclically oriented, if
  $(x,y)\in \Sigma$, then the arrow $y\to x$ in $Q_{\widetilde C}$
  corresponding to the one-dimensional vector space ${\rm
    Ext}^2_C(S_x,S_y)$ is the unique path from $x$ to $y$ in $\widetilde
  Q$ (see \cite{MR3169475}). In particular
  $e_y\cdot {\rm rad}(E)\cdot e_x=0$. Hence $e_yEe_x=e_yE'e_x$ or
  $e_yEe_x=e_yE''e_x$ according to whether $(x,y)\in \Sigma'$ or
  $(x,y)\in \Sigma''$.

  For every couple $(x,y)\in \Sigma$, let $\alpha_{(x,y)}\colon y\to x$
  be the corresponding arrow in $\widetilde Q$,
  let $r_{(x,y)}\in e_x\k Qe_y$ be a corresponding generator of $I$, and
  let $\xi_{(x,y)}\in {\rm Ext}^2_C(I_x,P_y)$ be a corresponding element
  in ${\rm Ext}^2_C(DC,C)$. Therefore we have
  \begin{enumerate}[(i)]
  \item $W=\sum_{(x,y)\in \Sigma} \alpha_{(x,y)}r_{(x,y)}$, 
  \item $E'$ is generated by $\{\xi_{(x,y)}\ |\ (x,y)\in
    \Sigma'\}$, and
  \item $E''$ is generated by $\{\xi_{(x,y)}\ |\ (x,y)\in
    \Sigma''\}$.
  \end{enumerate}
  Let $W'=\sum_{(x,y)\in \Sigma'}\alpha_{(x,y)}r_{(x,y)}$ and
  $W''=\sum_{(x,y)\in \Sigma''}\alpha_{(x,y)}r_{(x,y)}$. Hence
  $W=W'+W''$. To prove that this is a direct sum decomposition of $W$
  inducing the direct sum decomposition $E=E'\oplus E''$, it suffices
  to prove that no arrow of $\widetilde Q$
  appears simultaneously in a cycle of $W'$ and in a cycle of
  $W''$.
  
  By contradiction, assume there exists an arrow $a$ appearing simultaneously
  in a cycle of $W'$ and in a cycle of $W''$. Because of the definition of $W'$ and $W''$, the arrow $a$ is distinct from any
  $\alpha_{(x,y)}$, for $(x,y)\in \Sigma$. Therefore we have
  \begin{equation}
    \label{eq:1}
    \begin{array}{rcl}
      \partial_aW
      & = &
            \partial_aW' + \partial_aW''\\
      & = &
            \sum_{(x,y)\in \Sigma'} \varphi_{(x,y)} \alpha_{(x,y)}\psi_{(x,y)} +
            \sum_{(x,y)\in \Sigma''} \varphi_{(x,y)} \alpha_{(x,y)} \psi_{(x,y)}
    \end{array}
  \end{equation}
  where, in the second row, $\varphi_{(x,y)}$ and $\psi_{(x,y)}$ denote elements in $\k
  Q$. Note that each one of the two terms of this row is nonzero in $\k
  \widetilde Q$ because
  $\Sigma'$ and $\Sigma''$ are nonempty. Since $\partial_aW\in
  \widetilde I$, the expression (\ref{eq:1}) yields that
  \[
  \sum_{(x,y)\in \Sigma'} \varphi_{(x,y)} \alpha_{(x,y)}\psi_{(x,y)} +
  \widetilde I = 
  \sum_{(x,y)\in \Sigma''} \varphi_{(x,y)} \alpha_{(x,y)} \psi_{(x,y)}
  + \widetilde I\
  \]
  where the left-hand side  lies in $E'$ and the right-hand side
  lies in $E''$. Since $E'\cap E''=0$ it follows that both terms
\[  
\begin{array}{ll}
  \sum_{(x,y)\in \Sigma'} \varphi_{(x,y)} \alpha_{(x,y)}\psi_{(x,y)} &
  \text{and}\\
  \sum_{(x,y)\in
  \Sigma''} \varphi_{(x,y)} \alpha_{(x,y)} \psi_{(x,y)}
\end{array}
\]
are nonzero and lie in
  $\widetilde I$. Considering (c) above, both are nontrivial linear
  combinations of partial derivatives of $W$ with respect to arrows
  parallel to $a$. This contradicts (c). Thus the decomposition
  $W=W'\oplus W''$ is direct.
\end{proof}

Moreover, in the present situation, the direct sum decompositions of
the relation bimodule assume particularly nice forms.
\begin{cor}
  \label{sec:from-direct-decomp-6}
  Let $\widetilde C$ be a cluster tilted algebra with cyclically
  oriented quiver. Assume $E=E'\oplus E''$ is a direct sum
  decomposition. Then there exist direct sum decompositions
  $C_C=P'\oplus P''$ and $D(C)_C=I'\oplus I''$ such that $E'={\rm
    Ext}^2_{C}(I',P')$ and $E''={\rm Ext}^2_C(I'',P'')$.
\end{cor}
\begin{proof}
  As explained in the proof of
  Proposition~\ref{sec:from-direct-decomp-3}, given vertices $x,y$,
  the vector space ${\rm Ext}^2_C(D(Ce_x),e_yC)$ has dimension $0$ or
  $1$. The claimed decompositions of $DC$ and $C$ follow from this property.
\end{proof}

Note that the corollary implies that ${\rm Ext}^2_C(I'',P')={\rm Ext}^2_C(I',P'')=0$.

\begin{exple}
  Let $C$ be the tilted algebra given by the quiver
  
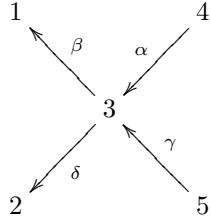
\[
  \xymatrix{
    1 & & 4 \ar[ld]_\alpha \\
    & 3 \ar[lu]_\beta  \ar[ld]^\delta& \\
    2 &  & 5 \ar[lu]_\gamma
  }
  \]
  bound by $\alpha \beta=0$, $\gamma\delta=0$. It is easily verified
  that $E={\rm Ext}^2_C(DC,C) = {\rm Ext}^2_C(I_4,P_1) \oplus {\rm
    Ext}^2_C(I_5,P_2)$. Moreover $\widetilde C$ is given by the quiver 
  
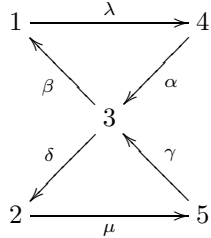
\[
  \xymatrix{
    1 \ar[rr]^\lambda & & 4 \ar[ld]^\alpha \\
    & 3 \ar[lu]^\beta \ar[ld]_\delta \\
    2 \ar[rr]_\mu & & 5 \ar[lu]_\gamma
  }
  \]
  with potential  $W=\alpha\beta\lambda +\gamma\delta\mu$. As seen in
  example (1.b) of \ref{sec:decomp-potent-exampl}, this is a direct sum
  decomposition of the potential $W$. It is easily seen that
  it corresponds to the direct sum decomposition $E=E'\oplus
  E''$ with the summand $\alpha \beta \lambda$ corresponding to
  $E'={\rm Ext}^2_C(I_4,P_1)$ and $\delta\gamma\mu$ corresponding to
  $E''= {\rm
    Ext}^2_C(I_5,P_2)$.
\end{exple}

\subsection{Induced decompositions of the potential: the $\widetilde{\mathbb A}$
  case}
\label{sec:from-direct-decomp-4}

Another case where the Keller potential is known to induce a system of
minimal relations is the case of cluster tilted algebras of type  $\widetilde{\mathbb A}$ (see
\cite{MR2592019}). Therefore, in this case also we can deduce a decomposition
of the Keller potential starting from a decomposition of the relation
bimodule. The
proof is  different from that of the cyclically oriented case. It
relies on the fact that the cluster tilted algebra $\widetilde C=\k
\widetilde Q/\widetilde I$ is gentle and on the following specific
combinatorial properties of 
$\widetilde Q$.
\begin{lem}
  \label{sec:from-direct-decomp-5}
  Let $i,j$ be vertices such that there exists an arrow
   $\alpha\colon i\to j$  in $\widetilde Q\backslash
  Q$ and such that $e_i{\rm rad}(E)e_j\neq 0$. Consider a path $u\beta v$
  from $i$ to $j$
  such that $u,v$ lie in $Q$ and are not both
  trivial, such that $\beta\colon i'\to j'$ is an
  arrow in $\widetilde Q\backslash Q$ and such that the class of
  $u\beta v$ in $e_i{\rm
    rad}(E) e_j$ is nonzero. Then
  \begin{enumerate}
  \item no arrow is parallel to $\beta$ (or $\alpha$),
  \item $\alpha$ and $u\beta v$ are the only paths in $\widetilde Q$
    not lying in $\widetilde I$, in particular $e_iEe_j$ is generated
    by $\alpha +\widetilde I$ and $u\beta v+\widetilde I$ and $e_i{\rm
      rad}(E)e_j$ is generated by $u\beta v+\widetilde I$, and
  \item  $e_{i'}{\rm rad}(E)e_{j'}=0$.
  \end{enumerate}
\end{lem}
\begin{proof}
  (1) Should $\alpha$ have a parallel arrow $\alpha'$, that arrow
  would lie in $\widetilde Q_1\backslash Q_1$. Since $(\widetilde
  Q,\widetilde I)$ is a gentle bound quiver, the path $u\beta v$ would
  start with $\alpha$ or $\alpha'$ and end with $\alpha$ and
  $\alpha'$. The path $u\beta v$ would therefore contain two arrows
  from $\widetilde Q_1\backslash Q_1$ instead of only one, namely $\beta$. This
  proves that no arrow is parallel to $\alpha$.

  By contradiction, assume that $\beta$ has a parallel arrow
  $\beta'$. Then $\beta'$ lies in $\widetilde Q_1\backslash Q_1$. Moreover
  $(\widetilde Q,\widetilde I)$ contains the following bound quivers
  
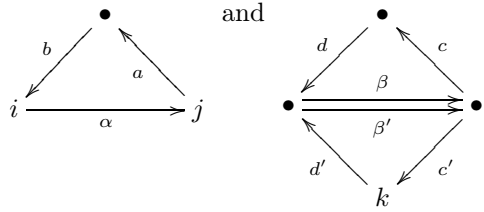
\[
  \xymatrix{
    &  \bullet \ar[ld]_b\\
    i \ar[rr]_{\alpha} & & j \ar[lu]^a
  }\ \text{and}\
  \xymatrix{
    & \bullet \ar[dl]_d & \\
    \bullet \ar@<2pt>[rr]^{\beta} \ar@<-2pt>[rr]_{\beta'} & & \bullet
    \ar[lu]_c \ar[ld]^{c'}\\
    & k \ar[lu]^{d'} & 
  }
  \]
  with relations all paths of length $2$ in any triangle. Moreover,
  there exist paths $u'$ and $v'$ in $Q$ with
  sources $i$ and $k$, respectively, and with targets $k$ and $j$,
  respectively such that $u=u'd'$ and $v=c'v'$, and hence $u\beta
  v=u'd'\beta c'v'$. As a consequence, $\widetilde C$ contains the
  following two full subcategories that are hereditary of type
  $\widetilde{\mathbb A}$
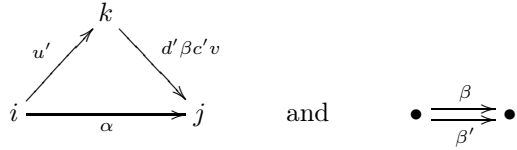
  \[
  \xymatrix{
    & k \ar[rd]^{d'\beta c'v}\\
    i \ar[ru]^{u'} \ar[rr]_{\alpha} & & j 
    & \text{and} & 
    \bullet \ar@<2pt>[r]^\beta \ar@<-2pt>[r]_{\beta'} & \bullet\
  }
  \]
  Note that these subcategories are indeed full because $(\widetilde
  Q,\widetilde I)$ is a gentle bound quiver. The existence of these two
  subcategories is a contradiction to the characterisation of
  cluster tilted algebras of type $\widetilde{\mathbb A}$,
  see \cite{MR2592019}.

  \medskip

  (2) This follows from the fact that $(\widetilde Q,\widetilde I)$ is a
  gentle  bound quiver.

  \medskip

  (3) 
  There only remains to prove that $e_{i'}{\rm rad}(E)e_{j'}=0$. If this
  was not the case, there would exist a path $w$ parallel to $\beta$, 
  not lying in $Q$, and such that $w\not\in \widetilde I$. According to
  (2), the paths $\beta$ and $w$ would be the only paths  in $\widetilde
  Q$ from $i'$ to $j'$. Hence $\widetilde C$ would have the following
  two full subcategories $\xymatrix{i' \ar@<2pt>[r]^{w} \ar@<-2pt>[r]_{\beta} &
    j'}$ and $\xymatrix{i \ar@<2pt>[r]^{u\beta v} \ar@<-2pt>[r]_{\alpha} &
    j}$. These are hereditary categories of type $\widetilde{\mathbb
    A}$. This would again contradict the classification of cluster tilted
  algebras of type $\widetilde{\mathbb A}$, see \cite{MR2592019}. Thus
  $e_{i'}{\rm rad}(E)e_{j'}=0$.
\end{proof}

Here is the construction of direct sum decomposition of the potential $W$
starting from direct sum decompositions of $E$ in the case of
cluster tilted algebras of type $\widetilde{\mathbb A}$.
\begin{prop}
  \label{sec:from-direct-decomp-7}
  Let $\widetilde C$ be a cluster tilted algebra of type
  $\widetilde{\mathbb A}$. Assume $E=E'\oplus E''$ is a direct sum
  decomposition of $E$ as $C-C$-bimodule. Then there exists a
  direct sum decomposition $W=W'\oplus W''$ of the Keller potential such that $E'$, $E''$ are
  respectively the partial relation bimodules corresponding to $W', W''$.
\end{prop}
\begin{proof}
  Let $\Sigma$ be the set of couples of vertices $(x,y)$ such that
  $e_x{\rm top}(E)e_y\neq 0$. Note that $e_x{\rm top}(E)e_y$ has
  dimension at most $2$ for any couple of vertices $(x,y)$ because
  $(\widetilde Q,\widetilde I)$ is a gentle bound quiver. According to
  the preceding lemma, the set
  $\Sigma$ admits the  partition $\Sigma=\Sigma_1\cup
  \Sigma_2\cup \Sigma_3$  where
  \begin{itemize}
  \item $\Sigma_1$ is the set of couples $(x,y)$ such that $e_x{\rm
      rad}(E)e_y=0$ and $e_x{\rm top}(E)e_y$ has dimension $1$,
  \item $\Sigma_2$ is the set of couples $(x,y)$ such that $e_x{\rm
      rad}(E)e_y\neq 0$,
  \item $\Sigma_3$ is the set of couples $(x,y)$ such that $e_x{\rm
      rad}(E)e_y=0$ and $e_x{\rm top}(E)e_y$ has dimension $2$.
  \end{itemize}
  In what follows we make a detailed study of these sets. Note that if
  $(i,j)\in \Sigma_1$ then ${\rm dim}(e_iEe_j)=1$. Therefore, $(i,j)\in
  \Sigma_1$ implies that
  \begin{equation}
    \label{eq:2}
    \begin{array}{lcl}
      \left\{
      \begin{array}{rcl}
        e_iEe_j & = & e_iE'e_j\\
        0 & =& e_iE''e_j
      \end{array}\right. &
                           \text{or} &
                                       \left\{
                                       \begin{array}{rcl}
                                         e_iEe_j & = & e_iE''e_j\\
                                         0 & =& e_iE''e_j
                                       \end{array}\right.
    \end{array}
  \end{equation}
  \medskip

  Now let us study $\Sigma_2$. According to
  Lemma~\ref{sec:from-direct-decomp-5}, and using the same notation, we
  have that $e_iEe_j$ is generated by $\alpha+\widetilde I$ and $u\beta
  v+\widetilde I$. Denote by $i'$ and $j'$ the source and
  target of $\beta$, respectively. Following
  Lemma~\ref{sec:from-direct-decomp-5}, the couple $(i',j')$ lies in 
  $\Sigma_1$. Without loss of generality we may assume that 
  $e_{i'}Ee_{j'}= e_{i'}E'e_{j'}$ and $e_{i'}E''e_{j'}=0$ (see (\ref{eq:2})). Assume that
  $\alpha + \widetilde I$ does not lie in $E'\cup E''$. Then
  there exists $\lambda\in \k^\times$ such that
  \[
  \alpha + \widetilde I = \left(\lambda\,u\beta v+
    \widetilde I\right) +
  \left( (\alpha-\lambda\,u\beta v)+ \widetilde I\right)
  \]
  is the decomposition of $\alpha\ {\rm  mod}\ \widetilde I$ according
  to $E=E'\oplus E''$. By construction of $\widetilde C$, the gentle
  bound quiver $(\widetilde Q,\widetilde I)$ contains a bound quiver of
  the following shape
  
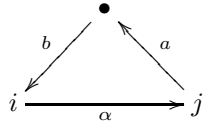
\[
  \xymatrix{
    & \bullet \ar[ld]_b \\
    i \ar[rr]_{\alpha} && j \ar[lu]_a
  }
  \]
  bound by $ab\in \widetilde I$, $b\alpha\in \widetilde I$, $\alpha a\in
  \widetilde I$. Therefore $u \beta va\not \in\widetilde I$
  because the last arrow of $u\beta v$ is not $\alpha$. Hence
  $\alpha-\lambda\, u\beta v + \widetilde I$ is an element of the
  $C-C$-bimodule $E''$ satisfying
  \[
  \left( \alpha-\lambda\,u\beta + \widetilde I\right)\cdot
  \left(a + \widetilde I\right) = \lambda u\beta va + 
  \widetilde I\in E'\backslash\{0\}\,.
  \]
  Remember that $\beta + \widetilde I\in E'$ by
  hypothesis. This contradicts the fact that the decomposition
  $E=E'\oplus E''$ is direct. Thus, $(i,j)\in \Sigma_2$ implies that
  \begin{equation}
    \label{eq:3}
    \alpha+ \widetilde I \in
    E'\cup E''
  \end{equation}
  where $\alpha\colon i\to j$ is the unique arrow of $\widetilde Q$
  with source $i$ and target $j$. As a consequence, exactly one the
  following situations occurs when $(i,j)\in \Sigma_2$:
  \begin{enumerate}[(a)]
  \item 
    $e_i E' e_j = {\rm Span}(\alpha + \widetilde
    I,\,u\beta v  +\widetilde I)$
    and 
    $e_i E''e_j=0$
  \item 
    $e_iE'e_j={\rm Span}(\alpha + \widetilde
    I)$
    and 
    $e_iE''e_j = {\rm Span}(u\beta v +
    \widetilde I)$
  \item 
    $e_iE'e_j = {\rm Span}(u\beta v + \widetilde I)$
    and 
    $e_iE''e_j = {\rm Span}(\alpha + \widetilde I)$
  \item 
    $e_iE'e_j=0$
    and 
    $e_iE''e_j = {\rm Span}(\alpha +
    \widetilde I,\,u\beta v + \widetilde
    I)$.
  \end{enumerate}
  
  \medskip

  Let us finally consider a couple $(i,j)\in \Sigma_3$. Then $e_i{\rm
    rad}(E)e_j=0$ and $(\widetilde Q,\widetilde I)$ contains a bound
  subquiver of the following shape
  
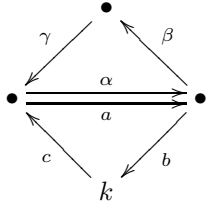
\[
  \xymatrix{
    & \bullet \ar[dl]_\gamma & \\
    \bullet \ar@<2pt>[rr]^{\alpha} \ar@<-2pt>[rr]_{a} & & \bullet
    \ar[lu]_\beta \ar[ld]^{b}\\
    & k \ar[lu]^{c} & 
  }
  \]
  with relations $\alpha\beta, \beta\gamma, \gamma\alpha, ab, bc,
  ca\in \widetilde I$. Denote by $\overline u$ the class modulo $\widetilde I$
  of a path $u$. Therefore $e_iEe_j = {\rm Span}(\overline
  a,\overline\alpha)$. Let us prove that $e_iE'e_j$ and $e_iE''e_j$
  are one of the subspaces $0$, ${\rm Span}(\overline a)$, ${\rm
    Span}(\overline \alpha)$ or ${\rm Span}(\overline
  a,\overline\alpha)$. If this is not the case, then there  exists
  an invertible matrix $\left(
    \begin{smallmatrix}
      t_1 & t_2\\ t_3 & t_4
    \end{smallmatrix}\right)$ such that
  \[
  \begin{array}{l}
    t_1\,\overline a +
    t_2\,\overline\alpha\in E' \\
    t_3\,\overline a + t_4\,\overline \alpha \in E''\\
    \text{with
    $t_1,t_2,t_3,t_4\in \k^\times$.}
  \end{array}
  \]
  This  implies that 
  \[
  \begin{array}{l}
    0\neq t_1\,\overline{a\beta} = (t_1\,\overline a+
    t_2\,\overline{\alpha}) \overline \beta\in E'\\
    0\neq t_3\,\overline{a\beta} = (t_3\,\overline a+ t_4\,\overline
    \alpha)\overline \beta\in E''\,.
  \end{array}
  \]
  This is absurd. Thus, if $(i,j)\in \Sigma_3$, then 
  \begin{equation}
    \label{eq:5}
    e_iE'e_j,\,e_iE''e_j\in \{0,{\rm
      Span}(\overline \alpha),\,{\rm Span}(\overline a),\,{\rm
      Span}(\overline \alpha,\overline a)\}\,.
  \end{equation}

  \medskip

  This study allows us to describe the claimed
  decomposition of $W$. Denote by $F$ the set of arrows in $\widetilde
  Q_1\backslash Q_1$. For every $\alpha\in F$ let $a_{\alpha}b_{\alpha}\in I$ be
  the associated monomial relation of length $2$ in $(Q,I)$. Thus
  $W=\sum_{\alpha\in F} \alpha a_{\alpha}b_{\alpha}$. Remember that if
  $\alpha,\beta$ are distinct arrows lying in $F$, then $\alpha
  a_{\alpha}b_{\alpha}$ and $\beta a_{\beta}b_{\beta}$ have no common
  arrow because $(\widetilde Q,\widetilde I)$ is a gentle bound
  quiver. It follows from  (\ref{eq:2}), (\ref{eq:3}), (a), (b), (c), (d), and
  (\ref{eq:5}) that $\overline \alpha\in E'$ or $\overline \alpha\in
  E''$, for every $\alpha\in F$. Denote by $F'$ and $F''$ the subsets of
  $F$ consisting of the arrows $\alpha\in F$ such that $\overline
  \alpha\in E'$ or $\overline\alpha\in E''$, respectively. This provides
  a partition $F=F'\cup F''$. Moreover, the $C-C$-bimodules  $E'$ and
  $E''$ are generated by 
  the classes modulo $\widetilde I$ of the arrows lying in $F'$ and
  $F''$, respectively. Let $W'=\sum_{\alpha\in F'}\alpha
  a_{\alpha}b_{\alpha}$ and $W''=\sum_{\alpha\in F''} \alpha a_{\alpha}
  b_{\alpha}$. The previous considerations show that $W=W'+W''$ is a
  direct sum decomposition that fits the requirements of the proposition.
\end{proof}

We now give an example showing that the analog of
Corollary~\ref{sec:from-direct-decomp-6} does not hold true for
cluster tilted 
algebras of type $\widetilde{\mathbb A}$.
Assume that there exist decompositions $C=P'\oplus P''$ and
$DC=I'\oplus I''$ such that $E'={\rm Ext}^2_C(I',P')$ and $E''={\rm
  Ext}^2_C(I'',P'')$. Then, for any pair $(x,y)$ of points in $Q$, we
have either $e_xE'e_y=0$ or $e_xE''e_y=0$.

\begin{exple}
  Let $C$ be given by the quiver
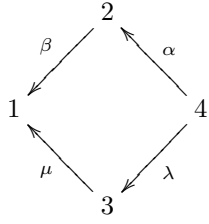
\[
\xymatrix{
  & 2 \ar[ld]_\beta \\
  1 & & 4\ar[lu]_\alpha \ar[ld]^\lambda \\
  & 3 \ar[lu]^\mu
}
\]
bound by all paths of length $2$. Then $\widetilde C$ is given by the
quiver
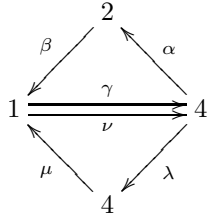
\[
\xymatrix{
  & 2 \ar[ld]_\beta \\
  1 & & 4\ar[lu]_\alpha \ar[ld]^\lambda \ar@{<-}@<2pt>[ll]^\nu \ar@{<-}@<-2pt>_\gamma[ll]\\
  & 4 \ar[lu]^\mu
}
\]
and the Keller potential is given by $W=\alpha \beta\gamma+\lambda\mu\nu$. The
summands $\alpha\beta\gamma$ and $\lambda\mu\nu$ are independent,
therefore the sum is direct and it induces a direct sum $E=E'\oplus
E''$ where $E'={\rm Span}(\gamma,\,\gamma\lambda,\,\mu\gamma,\,
\mu\gamma\lambda)$ and $E''={\rm Span}(
\nu,_,\nu\alpha,\,\beta\nu,\, \beta\nu\alpha)$. However we
have
$e_1E'e_4\neq 0$ and $e_1E''e_4\neq 0$. This shows that
Lemma~\ref{sec:from-direct-decomp-6} does not hold true in this
case.
\end{exple}

\section{Partial relation extension algebras}
\label{sec:part-relat-extens}

\subsection{The definition and examples}
\label{sec:1}

Let $C$ be a triangular algebra of global dimension at most $2$ and
$E'$ be a direct summand of the $C-C$-bimodule $E={\rm
  Ext}^2_C(DC,C)$. We recall that $\widetilde C=C\ltimes E$ is the
relation extension of $C$. Then the trivial extension $B=C\ltimes E'$ is called
the \emph{partial relation extension} of $C$ by $E'$. In this
subsection we prove a variant of transitivity for this construction. Let
$E=E'\oplus E''$ be a direct sum decomposition of the $C-C$-bimodule 
$E$ and $B=C\ltimes E'$. Denote by $\pi\colon B\to C$ the canonical
projection. Then $E''$ admits a $B-B$-bimodule structure by setting
\[
b_1x''b_2 = \pi(b_1)x'' \pi(b_2)
\]
for $b_1,b_2\in B$ and $x''\in E''$. 
\begin{lem}
\label{sec:definition-examples}
  With the preceding notation we have $\widetilde C=B\ltimes E''$.
\end{lem}
\begin{proof}
  We have an isomorphism of vector spaces:
  \[
  \begin{array}{crcl}
    \varphi \colon & C\ltimes E & \to & (C\ltimes E') \ltimes E''\\
                   & (c,e'+e'') & \mapsto &((c,e'),e'')\,.    
  \end{array}
  \]
  where $c\in C$, $e'\in E'$ and $e''\in E''$.
  It is necessary to check that
  \[\varphi((c_1,e_1'+e_1'')(c_2,e_2'+e_2'')) =
  \varphi(c_1,e_1'+e_1'')\varphi(c_2,e_2'+e_2'')\,.\]
  Indeed
  \[
  \begin{array}{rcl}
    \varphi(c_1,e_1'+e_1'')\varphi(c_2,e_2'+e_2'')
    & = &
          ((c_1,e_1'),e_1'')((c_2,e_2'),e_2'') \\
    & = &
          ((c_1,e_1')(c_2,e_2'), (c_1,e_1')e_2''+e_1''(c_2,e_2')) \\
    & = &
          ((c_1c_2,e_1'c_2+c_1e_2'),c_1e_2''+e_1''c_2) \\
    & = &
          \varphi(c_1c_2,e_1'c_2+e_1''c_2 + c_1e_2'+c_1e_2') \\
    & = &
          \varphi((c_1,e_1'+e_1'')(c_2,e_2'+e_2''))\,.
  \end{array}
  \]
\end{proof}

We pose the following problem on the meaning of $E''$ in terms of
$C\ltimes E'$.

\begin{pb}
  Let $C$ be a triangular algebra of global dimension at most $2$ and
  $E=E'\oplus E''$ be a direct sum decomposition of the $C-C$-bimodule
  $E={\rm Ext}^2_C(DC,C)$. What is the connection between $E''$ and
  the relation bimodule of the partial relation extension $C\ltimes E'$?
\end{pb}

\begin{rem}
  We may define a poset of partial relation extensions. We say that
  $B_1=C\ltimes E_1$ is \emph{smaller} than $B_2=C\ltimes E_2$ if
  $E_1$ is a direct summand of $E_2$. Notice that the obtained poset
  admits $\widetilde C$ as a unique maximal element and it admits $C$ as
  a unique minimal element. This poset is infinite in general. For instance, let $C$ be the algebra given
  by the following quiver
  
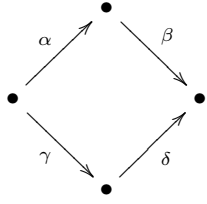
\[
  \xymatrix{
    & \bullet  \ar[rd]^\beta \\
    \bullet \ar[ru]^\alpha \ar[rd]_\gamma & & \bullet\\
    & \bullet \ar[ru]_\delta
  }
  \]
  and relations $\alpha\beta,\gamma\delta$. Then, ${\rm
    dim}_{\k}\,E=2$. Let $(u,v)$ be a 
  basis of $E$. For every point $[x:y]$ on the projective line
  $\mathbb P_1(\k)$
  denote by $B_{[x:y]}$ 
  the partial relation extension of $C$ by the one-dimensional
  subbimodule of $E$ generated by $x\,u+y\,v$. The resulting partial
  relation extensions are pairwise isomorphic. Then, the poset consists
  of the algebras $C$, $\widetilde C$ and $B_{[x:y]}$, for $[x:y]\in
  \mathbb P_1(\k)$, and it has the following shape
  
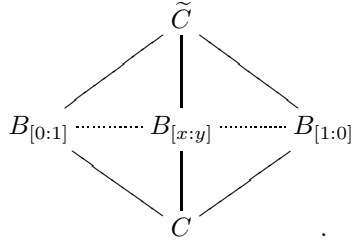
\[
  \xymatrix{
     &\widetilde C \\
     B_{[0:1]}\ar@{-}[ru] \ar@{-}[rd] \ar@{.}[r]
    &
    B_{[x:y]} \ar@{.}[r] \ar@{-}[u] \ar@{-}[d]
    & B_{[1:0]} \ar@{-}[lu] \ar@{-}[ld]\\
     & C& .
  }
  \]
\end{rem}

\subsection{The bound quiver of a partial relation
  extension}
\label{sec:bound-quiver-partial}

Let $C=\k Q/I$ be a triangular algebra of global dimension at most
two, let $\widetilde C=C\ltimes {\rm Ext}^2_C(DC,C)$ be its relation
extension, and assume that $E={\rm Ext}^2_C(DC,C)$ has a
$C-C$-bimodule direct sum decomposition $E=E'\oplus E''$. Our
objective is to describe a bound quiver presentation of the partial
relation extension $B=C\ltimes E'$ when this direct sum decomposition
arises from a direct sum decomposition of the Keller potential
associated with a minimal system of relations in $I$, see
proposition~\ref{sec:from-direct-decomp-2}.

Now, it follows from \cite[(2.4)]{MR2409188} that the new arrows generate the
top of the $C-C$-bimodule ${\rm Ext}^2_C(DC,C)$. Assume that there exists a direct sum
decomposition $W=W'\oplus W''$ of the Keller potential in such a way that
$E'$ and $E''$ are the partial relation bimodules corresponding to
$W'$ and $W''$ respectively, see
proposition~\ref{sec:from-direct-decomp-2}. Then the set of new arrows
can be partitioned into two sets $\{\alpha_1',\ldots,\alpha_s'\}$ and
$\{\alpha_1'',\ldots,\alpha_t''\}$ forming respectively the tops of
$E'$ and $E''$. We may now state.

\begin{cor}
  \label{sec:bound-quiver-partial-3}
  Let $C=\k Q/I$ be a triangular algebra of global dimension at most
  two, $\widetilde C$ its relation extension, $W$ the Keller potential
  associated with a minimal system of relations in $I$, and $\mathcal
  J$ the square of the ideal of $J(\widetilde Q,W)$ generated by the
  new arrows. If $E=E'\oplus E''$ is a direct sum decomposition of
  $C-C$-bimodule arising from a direct sum decomposition of the Keller
  potential, 
  $\alpha_1'',\ldots,\alpha_t''$ are the new arrows generating the
  top of $E''$ and $\mathcal J'=\mathcal J+\sum_{i=1}^t\widetilde
  C\alpha_i''\widetilde C$. Then 
  \[
  C\ltimes E'=J(\widetilde Q,W)/\mathcal J'\,.
  \]
\end{cor}
\begin{proof}
Let $B=C\ltimes E'$. It follows from
lemma~\ref{sec:definition-examples} that
$B\simeq \widetilde C/E''$. By definition,  $E''$ is the
subbimodule of ${\rm Ext}^2_C(DC,C)$ generated by the classes of the
new arrows $\alpha_1'',\ldots,\alpha_t''$, see
section~\ref{sec:from-direct-decomp}. Hence the statement follows from
the fact that $\widetilde C\simeq J(\widetilde Q,W)/\mathcal J$, see
\ref{sec:from-direct-decomp}. 
\end{proof}

Thus, $B$ is given by the bound quiver obtained from that of
$\widetilde C=\k \widetilde Q/\widetilde I$ by simply deleting the arrows $\alpha_i''$ from the
ordinary quiver and by deleting any path involving such an arrow from
any relation. Setting $W'=\sum_{i=1}^s\rho_i'\alpha_i'$ and
$W''=\sum_{i=1}^t\rho_i''\alpha_i''$ with $\alpha_i',\alpha_j''$ the
new arrows and $\rho_i',\rho_i''$ the elements of the chosen minimal
system of relations $R$ corresponding to $\alpha_i',\alpha_j''$
respectively, then the top of $E'$ is generated by the $\alpha_i'$ and
the top of $E''$ is generated by the $\alpha_j''$, so we can state the
following corollary.

\begin{cor}
  With the above notation, $B=C\ltimes E'$ has a bound quiver as
  follows
  \begin{enumerate}[(a)]
  \item $(Q_B)_o=Q_o=\widetilde Q_o$,
  \item $(Q_B)_1=\widetilde
    Q_1\backslash\{\alpha_1'',\ldots,\alpha_t''\}=Q_1\cup\{\alpha_1',\ldots,\alpha_s'\}$,
  \item the binding ideal $I_B$ is generated by the cyclic partial
    derivatives of $W'$, the relations $\rho_1'',\ldots,\rho_t''$ and
    $\mathcal J$.
  \end{enumerate}
\end{cor}

\begin{exple}
  Let $C$ be the tilted algebra given by the quiver
  
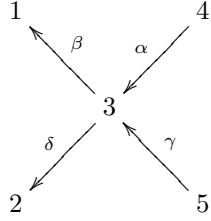
\[
  \xymatrix{
    1 && 4 \ar[ld]_\alpha \\
    & 3 \ar[lu]_\beta \ar[ld]_\delta &\\
    2 && 5 \ar[lu]_\gamma
  }
  \]
  bound by $\alpha\beta=0$, $\gamma\delta=0$. Then $\widetilde C$ is
  the Jacobian algebra given by the quiver
  
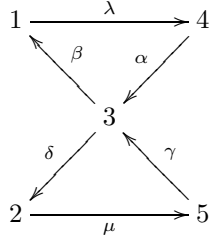
\[
  \xymatrix{
    1 \ar[rr]^\lambda&& 4 \ar[ld]_\alpha \\
    & 3 \ar[lu]_\beta \ar[ld]_\delta &\\
    2 \ar[rr]_\mu && 5 \ar[lu]_\gamma 
  }
  \]
  and the Keller potential $W=\alpha\beta\lambda+\gamma\delta\mu$. As
  seen in Section~\ref{sec:decomp-potent-exampl}, $W'=\alpha\beta\lambda$ and
  $W''=\gamma\delta\mu$ are independent so that $W=W'\oplus W''$ is a direct sum
  decomposition. Setting $E'={\rm Ext}^2_C(I_1,P_4)$ and $E''={\rm
    Ext}^2_C(I_2,P_5)$, then $E=E'\oplus E''$ is a direct sum
  decomposition of the bimodule $E={\rm Ext}^2_C(DC,C)$ corresponding
  to the previous decomposition of the potential. The algebra
  $B=C\ltimes E'$ is given by the quiver
  
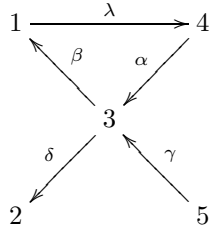
\[
  \xymatrix{
    1 \ar[rr]^\lambda&& 4 \ar[ld]_\alpha \\
    & 3 \ar[lu]_\beta \ar[ld]_\delta &\\
    2 && 5 \ar[lu]_\gamma 
  }
  \]
  bound by $\alpha\beta=0$, $\gamma\delta=0$, $\lambda\alpha=0$ and $\beta\lambda=0$.
\end{exple}

\subsection{The module category of a partial relation extension}
\label{sec:module-categ-part}

In the present subsection, we assume that $C$ is tilted, so that its
relation extension $\widetilde C$ is cluster tilted. Our objective is
to give two descriptions of the 
module category of a partial relation extension, one as a quotient of
a module category of a cluster tilted algebra, and the other as a
quotient of another category which we now define. We mean by module a
finitely generated 
right module. Given an algebra $B$ we denote by ${\rm mod}\,B$ its
module category.

We consider the
following setting. Let $A$ be a hereditary algebra, $\mathcal C_A$ the
corresponding cluster category and $T$ a cluster tilting object in
$\mathcal C_A$. We denote by $\mathcal D^b({\rm mod}\,A)$ the bounded
derived category of ${\rm mod}\,A$ and by $\tau$ and $[-]$
respectively the Auslander-Reiten translation and the shift of
$\mathcal D^b({\rm mod}\,A)$ respectively. Because of \cite[Theorem
3.3]{MR2249625} we may assume that $T$ is 
actually a tilting module over $A$. We denote by $C={\rm End}_A(T)$
the tilted algebra and by $\widetilde C={\rm End}_{\mathcal C_A}(T)$. Then
$\widetilde C$ is the relation extension of $C$.

We recall that it is
shown in \cite{MR2409188} that $E={\rm Ext}^2_C(DC,C)$ is isomorphic to
${\rm Hom}_{\mathcal
  D^b({\rm mod}\,A)}(T,\tau^{-1}\circ T[1])$ as a $C-C$-bimodule. Assume that $E=E'\oplus E''$ is a
$C-C$-bimodule direct sum decomposition. Observe that $E'$ and $E''$ can be
considered as subbimodules of ${\rm Hom}_{\mathcal
  D^b({\rm mod}\,A)}(T,\tau^{-1}\circ T[1])$ and the latter may in turn be considered as contained in ${\rm End}_{\mathcal
  C_A}(T)=\widetilde C$, see \cite{MR2409188}.

Let $\mathcal I$ be the
ideal of all morphisms in $\mathcal C_A$ generated by $E''$ that is,
of all morphisms of $\mathcal C_A$ which factor through an element of
$E''$ considered as a morphism from $T$ to $T$.
We define $\mathcal B$ to be the additive quotient category $\mathcal
C_A$ by $\mathcal I$, that is, $\mathcal B$ has the same objects as
those of $\mathcal C_A$ and, if $X,Y$ are two such objects, then ${\rm
  Hom}_{\mathcal B}(X,Y)= {\rm Hom}_{\mathcal C_A}(X,Y)/\mathcal I(X,Y)$.

\begin{prop}
  \label{sec:module-categ-part-1}
  With the above notation  ${\rm
    End}_{\mathcal B}(T)$ is isomorphic to the partial relation
  extension $B=C\ltimes E'$.
\end{prop}
\begin{proof}
  Because $\mathcal B=\mathcal C_A/\mathcal I$, we have ${\rm
    End}_{\mathcal B}(T)= {\rm End}_{\mathcal C_A}(T)/\mathcal
  I(T,T)$. However, as ideals of ${\rm End}_{\mathcal C_A}(T)$ we have
  $E''=\mathcal I(T,T)$. Hence ${\rm End}_{\mathcal B}(T)\simeq {\rm
    End}_{\mathcal C_A}(T)/E''\simeq \left( C\ltimes (E'\oplus
    E'')\right)/E''\simeq C\ltimes E'$.
\end{proof}

As a corollary, for every object $X$ in $\mathcal B$, the ${\rm
  End}_{\mathcal B}(T)$-module ${\rm Hom}_{\mathcal B}(T,X)$ is a
$B$-module. Thus we have a functor ${\rm Hom}_{\mathcal B}(T,-)\colon
\mathcal B\to {\rm mod}\,B$, which is full and dense. More precisely,
we have the following lemma.
\begin{lem}
  \label{sec:module-categ-part-4}
  We have a commutative diagram of full and dense functors
  \[
  \xymatrix{
    \mathcal C_A \ar@{->>}[rr]^{{\rm Hom}_{\mathcal C_A}(T,-)} \ar@{->>}[d]_{\pi} && {\rm
      mod}\,\widetilde C \ar@{->>}[d]^{-\underset{\widetilde C}{\otimes}B}\\
    \mathcal B \ar@{->>}[rr]_{{\rm Hom}_{\mathcal B}(T,-)} && {\rm mod}\,B
  }
  \]
  where $\pi\colon \mathcal C_A\to \mathcal B=\mathcal C_A/\mathcal I$ is the
  canonical projection.
\end{lem}
\begin{proof}
  The functor $-\underset{\widetilde C}{\otimes} B$ maps a $\widetilde
  C$-module $M$ to the $B$-module
  \[
  M\underset{\widetilde
    C}{\otimes}B=M\underset{\widetilde C}{\otimes} \widetilde
  C/E''\simeq M/ME''\,.
  \]
  Thus \[
  (-\underset{\widetilde C}{\otimes} B)\circ
  {\rm Hom}_{\mathcal C_A}(T,-)(X)\simeq {\rm Hom}_{\mathcal
    C_A}(T,X)/{\rm Hom}_{\mathcal C_A}(T,X)E''\,.
  \]

  On the other hand
  \[
  \begin{array}{rcl}
    {\rm Hom}_{\mathcal B}(T,-) \circ \pi(X)
    & = &
          {\rm Hom}_{\mathcal
          B}(T,X)\\
    & = &
          {\rm Hom}_{\mathcal C_A}(T,X)/\mathcal I(T,X)\,.
  \end{array}
  \]
  Now notice that $\mathcal I(T,X)$ is the image of the morphism ${\rm
    Hom}_{\mathcal C_A}(T,X)\otimes E''\to \mathcal I(T,X)$ given by
  $u\otimes v\mapsto u\circ v$. Indeed, let $f\in \mathcal I(T,X)$, then
  $f=\sum_i u_i\circ e_i\circ v_i$ where $e_i\in E''$, $v_i\colon T\to
  E''$ and $u_i\colon E''\to X$. Because $\mathcal I(T,T)=E''$ is an
  ideal in ${\rm End}_{\mathcal C_A}(T)$, we have $e_i\circ v_i\in
  E''$. Therefore $f=\sum_iu_i\circ (e_i\circ v_i)$ belongs to the image
  of the given map. This shows that $\mathcal I(T,X)={\rm
    Hom}_{\mathcal C_A}(T,X)E''$. The shown diagram is thus commutative.

  Now, if $M$ is a $B$-module, then it admits a natural $\widetilde
  C$-module structure, and, with respect to this structure,
  $M\underset{\widetilde C}\otimes B\simeq M_B$. Thus the functor
  $-\underset{\widetilde C}\otimes B$ is full and dense. On the other
  hand,  ${\rm Hom}_{\mathcal
    C_A}(T,-)$ is  full and dense because of \cite[Proposition
  2.1]{MR2247893}. Hence ${\rm 
    Hom}_{\mathcal B}(T,-)$ is full and dense.
\end{proof}

We now turn our attention to the kernel of the composed functor
$(-\underset{\widetilde C}\otimes B)\circ {\rm Hom}_{\mathcal
  C_A}(T,-)\colon \mathcal C_A\to {\rm mod}\,B$.
\begin{lem}
  \label{sec:module-categ-part-2}
  The kernel  of the composed functor $(-\underset{\widetilde
    C}{\otimes}B)\circ {\rm Hom}_{\mathcal C_A}(T,-)$ is the ideal
  $\mathcal K$ of
  $\mathcal C_A$ consisting of all morphisms $f\colon X\to Y$ such that
  the composition of $f$ with a minimal ${\rm add}(T)$-approximation $u_X\colon
  T_X\to X$ can be written in the form $f\circ u_X=u_Y \circ e$ where $e\in
  E''$ and $u_Y\colon T_Y\to Y$ is a minimal ${\rm add}(T)$-approximation.
\end{lem}
\begin{proof}
  Let $f\colon X\to Y$ be a morphism in $\mathcal C_A$.
  Using minimal ${\rm add}(T)$-approximations of $X$ and $Y$
  yields the following diagram in $\mathcal C_A$.
  \[
  \xymatrix{
    T_X  \ar[r]^{u_X} &
    X \ar[d]^f \\
    T_Y \ar[r]^{u_Y} &
    Y \,.
  }
  \]
  The image of $f$ 
  in ${\rm mod}\,B$ is equal  to that of the mapping 
  \[
  {\rm Hom}_{\mathcal C_A}(T,X)/{\rm Hom}_{\mathcal C_A}(T,X)E''\to
  {\rm Hom}_{\mathcal C_A}(T,Y) / {\rm Hom}_{\mathcal C_A}(T,Y)E''
  \]
  given by $\overline u \mapsto \overline{f\circ u}$,  where the
  notation $\overline g$ stands for the residual class of a morphism
  $g$ in its respective quotient. If
  $\overline{f\circ u}$ vanishes for every $\overline u$ then it 
  vanishes for $u=u_X$. Because $\overline{f\circ u_X}=0$, there exist $T_0\in {\rm add}(T)$,
  $e_0\in E''$ and a morphism $g_0\colon T_0\to Y$ such that $f\circ u_X
  = g_0\circ e_0$. Because $u_Y$ is a minimal ${\rm
    add}(T)$-approximation, $g_0$ factors through it and thus there
  exists a morphism $g'\colon T_0\to T_Y$ such that $u_Y\circ
  g'=g_0$. Setting $e=g'\circ e_0$ we get that $e\in E''$ because the
  latter is an ideal and $u_Y\circ e=f\circ u_X$.
  \[
  \xymatrix{
    T_X \ar[d]_{e} \ar[r]^{u_X} \ar[d]_{}&
    X \ar[d]^f  \\
    T_Y \ar[r]^{u_Y} &
    Y \,.
  }
  \]
  This proves that $f$
  belongs to $\mathcal K$. Conversely if $f$ belongs to $\mathcal K$
  then it is immediate that its image in ${\rm mod}\,B$ is zero.
\end{proof}

\setcounter{thm}{1}
\begin{thm}
  \label{sec:module-categ-part-3}
The composed functor $(-\underset{\widetilde C}{\otimes} B) \circ
  {\rm Hom}_{\mathcal C_A}(T,-)$ induces an equivalence ${\rm
    mod}\,B\simeq \mathcal C_A/\mathcal K$.
\end{thm}
\begin{proof}
  This follows immediately from lemmata \ref{sec:module-categ-part-4}
  and \ref{sec:module-categ-part-2}.
\end{proof}

Note that, taking $E''$ equal to $0$ yields the main theorem of
\cite{MR2247893}.

This theorem entails several consequences. Let  $C$ be a tilted
algebra. Recall that the \emph{cluster repetitive} algebra is the
locally finite dimensional algebra without identity
\[
\check{C} =
\left(
  \begin{matrix}
    \ddots & 0& & &0\\
    \ddots & C_{-1} & 0& &\\
    0 & E_0 & C_0 & 0 \\
    & 0& E_1 & C_1 & \\
    0&  & 0& \ddots &\ddots
  \end{matrix}
\right)
\]
where the matrices have only finitely many nonzero  entries, $C_i=C$ and
$E_i={\rm Ext}^2_C(DC,C)$ for all $i\in\mathbb Z$, all remaining
entries are zero and the multiplication is induced from that of $C$,
the $C-C$-bimodule structure of ${\rm Ext}^2_C(DC,C)$ and the zero map
${\rm Ext}^2_C(DC,C)\otimes {\rm Ext}^2_C(DC,C)\to 0$. The identity
maps $C_i\to C_{i-1}$ and $E_i\to E_{i-1}$ induce an automorphism
$\varphi$ of $\widetilde C$ and the orbit category
$\check{C}/\langle \varphi\rangle$ inherits from
$\check{C}$ a $\k$-algebra structure isomorphic to $\widetilde
C=C\ltimes {\rm Ext}^2_C(DC,C)$. Thus the projection functor $G\colon
\check{C}\to \widetilde C$ is a Galois covering with infinite
cyclic group generated by $\varphi$. We denote by $G_{\lambda}\colon
{\rm mod}\,\check{C}\to {\rm mod}\,\widetilde C$ the
associated push-down functor (see \cite{MR654725}).

Now let $A$ be a hereditary algebra and $T$ be a tilting $A$-module
such that $C={\rm End}_A(T)$. Consider the automorphism
$F=\tau^{-1}\circ [1]$ in $\mathcal D^b({\rm mod}\,A)$ and let
$\pi'\colon \mathcal D^b({\rm mod}\,A)\to \mathcal C_A$ denote the
canonical projection onto the cluster category. We are now able to
state the first corollary.
\begin{cor}
  With the above notation, there exists a commutative diagram of full
  and dense functors
  \[
  \xymatrix{
    \mathcal D^b({\rm mod}\,A) \ar@{->>}[rrrr]^{
      {\rm Hom}_{\mathcal D^b({\rm mod}\,A)}(\oplus_{i\in \mathbb
        Z}F^iT,-)
    }
    \ar[d]_{\pi\pi'} &&&&
    {\rm mod}\,\check{C}
    \ar[d]^{
      (-\underset{\widetilde C}{\otimes} B) \circ G_{\lambda}
    } \\
    \mathcal B \ar[rrrr]_{
      {\rm Hom}_{\mathcal B}(\pi'T,-)
    } &&&&
    {\rm mod}\,B\,.
  }
  \]
\end{cor}
\begin{proof}
  It is shown in \cite[Theorem 9 of 2.3]{MR2497589}, that there is a
  commutative diagram of dense functors
  \[
  \xymatrix{
    \mathcal D^b({\rm mod}\,A) \ar@{->>}[rrrr]^{
      {\rm Hom}_{\mathcal D^b({\rm mod}\,A)}(\oplus_{i\in \mathbb
        Z}F^iT,-)
    }
    \ar[d]_{\pi'} &&&&
    {\rm mod}\,\check{C}
    \ar[d]^{
      G_{\lambda}
    } \\
    \mathcal C_A \ar[rrrr]_{
      {\rm Hom}_{\mathcal C_A}(\pi'T,-)
    } &&&&
    {\rm mod}\,\widetilde C\,.
  }
  \]
  These functors are also full: $\pi'$ is full by definition, ${\rm Hom}_{\mathcal D^b({\rm mod}\,A)}(\oplus_{i\in \mathbb
    Z}F^iT,-)$ is full because of \cite[Proposition 7 of
  2.1]{MR2497589}, and ${\rm Hom}_{\mathcal C_A}(\pi'T,-)$ is full
  because of \cite[Proposition 2.1]{MR2247893}. The required commutative square
  follows upon composing this diagram with the one of lemma
  \ref{sec:module-categ-part-4} above.
\end{proof}

As a consequence of this corollary, there is also a relation with
the repetitive algebra $\widehat C$ of $C$, this is the algebra
\[
\widehat C=
\left(
  \begin{matrix}
    \ddots & 0& & &0\\
    \ddots & C_{-1} & 0& &\\
    0 & Q_0 & C_0 & 0 \\
    & 0& Q_1 & C_1 & \\
    0&  & 0& \ddots &\ddots
  \end{matrix}
\right)
\]
where matrices have only finitely many nonzero entries, $C_i=C$ and
$Q_i=DC$ for all $i\in \mathbb Z$, all
remaining entries are zero, addition is the usual addition of matrices
and multiplication is induced from that of $C$, the $C-C$-bimodule
structure of $DC$ and the zero maps $DC\otimes DC\to 0$. The Nakayama
automorphism $\nu$ of $\widehat C$ is the one  induced by the identity
maps $C_i\to C_{i-1}$, $Q_i\to Q_{i-1}$. Then the quotient category
$\widehat C/\langle \nu\rangle$ is isomorphic to the trivial extension
$T(C)=C\ltimes DC$ of $C$ by its minimal injective cogenerator $DC$
(see \cite{MR693045}). There is a natural functor from ${\rm
  mod}\,\widehat C$ to ${\rm mod}\,\check C$: Indeed, let
$p\colon {\rm mod}\,\widehat C\to \underline{\rm mod}\,\widehat
C$ denote the canonical projection, and define $\Phi\colon {\rm
  mod}\,\widehat C\to {\rm mod}\,\check C$ to be the
composition
\[
{\rm mod}\,\widehat C \xrightarrow p \underline{\rm mod} \widehat C
\xrightarrow{
  \underline{\rm Hom}_{\widehat C}(
  \oplus_{i\in \mathbb Z} \tau^i\Omega^{-i}C,-
  )}
{\rm mod}\,\check C\,.
\]

\begin{cor}
  With the above notation, there exists a commutative diagram of full
  and dense functors
  \[
  \xymatrix{
    {\rm mod}\,\widehat C \ar@{->>}[d]_{\pi''} \ar@{->>}[rr]^{\Phi} &&
    {\rm mod}\,\check C \ar@{->>}[d]^{(-\underset{\widetilde C}\otimes
      B) \circ G_\lambda}\\ 
    \mathcal B \ar@{->>}[rr]^{{\rm Hom}_{\mathcal B}(\tau'T,-)} && {\rm
      mod}\,B\,.
  }
  \]
\end{cor}
\begin{proof}
  Let  $\mathcal C_C$ be  the orbit category of $\underline{\rm
    mod}\,\widehat C$ under the action of the automorphism
  $F_C\colon \underline{\rm mod}\,\widehat C\to \underline{\rm
    mod}\,\widehat C$ defined by $F_C=\tau^{-1}\Omega^{-1}$ and
  the morphism space from $(F_C^iX)_{i\in \mathbb Z}$ to
  $(F_C^iY)_{i\in \mathbb Z}$ is $\oplus_{i\in \mathbb
    Z}\underline{\rm Hom}_{\widehat C}(X,F_C^iY)$. Also let $\widehat
  \pi$ be the composition of the two projection functors $p\colon
  {\rm mod}\,\widehat C\to \underline{\rm mod}\,\widehat C$ and
  $\underline{\widehat \pi}\colon \underline{\rm mod}\,\widehat C
  \to \mathcal C_C$. Then, there is
  a commutative diagram of full and dense functors, see \cite[Theorem
  17 of 3.4]{MR2497589}:
  \[
  \xymatrix{
    {\rm mod}\,\widehat C  \ar@{->>}[rr]^\Phi
    \ar@{->>}[d]_{\widehat \pi}&& {\rm mod}\check C \ar@{->>}[d]^{G_\lambda} \\
    \mathcal C_C \ar@{->>}[rr]^{{\rm Hom}_{\mathcal C_C}(\widehat
      \pi C,-)} && {\rm mod}\widetilde C\,.
  }
  \]

  Moreover, it follows from \cite[Lemma 15 of 3.2]{MR2497589} that
  there is a commutative diagram of full and dense functors
  \[
  \xymatrix{
    \mathcal C_A\ar@{->>}[rd]^{{\rm Hom}_{\mathcal C_A}(\pi T,-)}
    \ar@{->>}[dd]_{\eta} \\
    & {\rm mod}\,\widetilde C \\
    \mathcal C_C \ar@{->>}[ru]_{{\rm Hom}_{\mathcal
        C_C}(\underline{\widehat \pi} C,-)}
  }
  \]
  with $\eta$ an equivalence.

  The required diagram follows upon composing these two diagrams
  with the one of lemma \ref{sec:module-categ-part-4} above. The
  functor $\pi''\colon {\rm mod}\,\widehat C\to \mathcal B$ is
  equal to the composition $\pi\circ \eta^{-1}\circ \widehat \pi$.
\end{proof}

\begin{exple}\ 
  \begin{enumerate}[(a)]
  \item Let $C$ be the tilted algebra given by the quiver
    
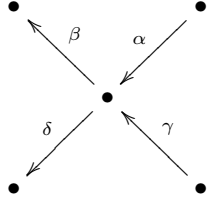
\[
    \xymatrix{
      \bullet & & \bullet \ar[ld]_{\alpha} \\
      & \bullet \ar[lu]_{\beta} \ar[ld]_\delta & \\
      \bullet & & \bullet \ar[lu]_\gamma
    }
    \]
    bound by $\alpha\beta=0$, $\gamma\delta=0$. Then its relation
    extension $\widetilde C$ is given by the quiver
    
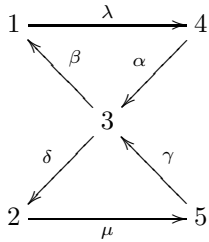
\[
    \xymatrix{
      1 \ar[rr]^\lambda && 4 \ar[ld]_\alpha \\
      & 3 \ar[lu]_\beta \ar[ld]_\delta \\
      2 \ar[rr]_\mu && 5 \ar[lu]_\gamma
    }
    \]
    and the potential $W=\alpha\beta\lambda+\gamma\delta\mu$. As
    seen before in section \ref{sec:decomp-potent-exampl}, this is
    a direct sum decomposition $W=W_1+W_2$ with
    $W_1=\alpha\beta\lambda$, $W_2=\gamma\delta\mu$. Let $E'$ be
    the direct summand of the $C-C$-bimodule $E={\rm
      Ext}^2_C(DC,C)$ corresponding to $W_1$. Then $B=C\ltimes E'$
    is given by the quiver
    
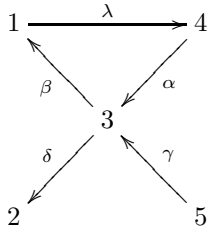
\[
    \xymatrix{
      1 \ar[rr]^\lambda && 4 \ar[ld]^\alpha \\
      & 3 \ar[lu]^\beta \ar[ld]_\delta \\
      2 && 5 \ar[lu]_\gamma
    }
    \]
    bound by $\alpha\beta=0$, $\beta\lambda=0$, $\gamma\delta=0$,
    $\lambda\alpha=0$. Its Auslander-Reiten quiver $\Gamma({\rm
      mod}\,B)$ is given by
    
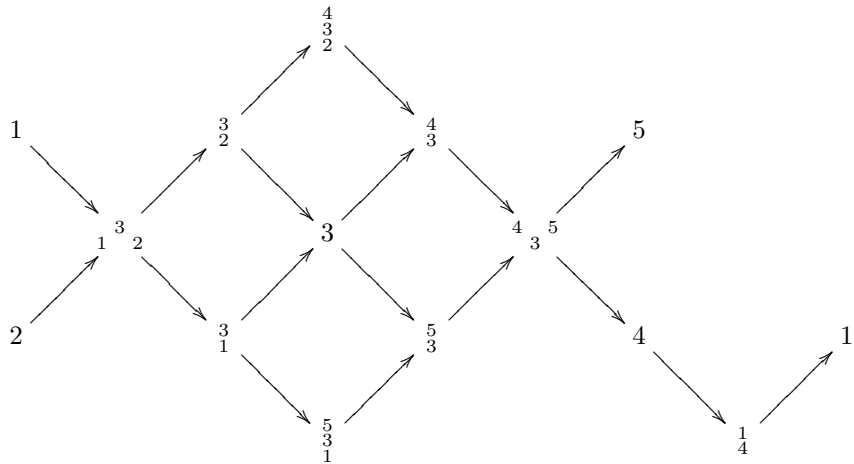
\[
    \xymatrix@!@=12pt{
      &&&
      {\begin{smallmatrix}
          4\\3\\2
        \end{smallmatrix}}\ar[rd]
      \\
      1\ar[rd] &&
      {\begin{smallmatrix}
          3\\2
        \end{smallmatrix}}\ar[rd]\ar[ru]
      &&
      {\begin{smallmatrix}
          4\\3
        \end{smallmatrix}}\ar[rd]
      &&
      5
      \\
      &
      {\begin{smallmatrix}
          &3&\\
          1&&2
        \end{smallmatrix}}\ar[rd]\ar[ru]
      &&
      3\ar[rd]\ar[ru]
      &&
      {\begin{smallmatrix}
          4&&5\\
          &3
        \end{smallmatrix}}\ar[rd]\ar[ru]
      \\
      2\ar[ru] &&
      {\begin{smallmatrix}
          3\\1
        \end{smallmatrix}}\ar[rd]\ar[ru]
      &&
      {\begin{smallmatrix}
          5\\3
        \end{smallmatrix}}\ar[ru]
      &&4\ar[rd]
      &&
      1\\
      &&&
      {\begin{smallmatrix}
          5\\3\\1
        \end{smallmatrix}}\ar[ru]
      &&&&
      {\begin{smallmatrix}
          1\\4
        \end{smallmatrix}}\ar[ru]
    }
    \]
    where the two copies of the simple $S_1=1$ are identified. The
    reader may compare this quiver with $\Gamma({\rm
      mod}\,\widetilde C)$
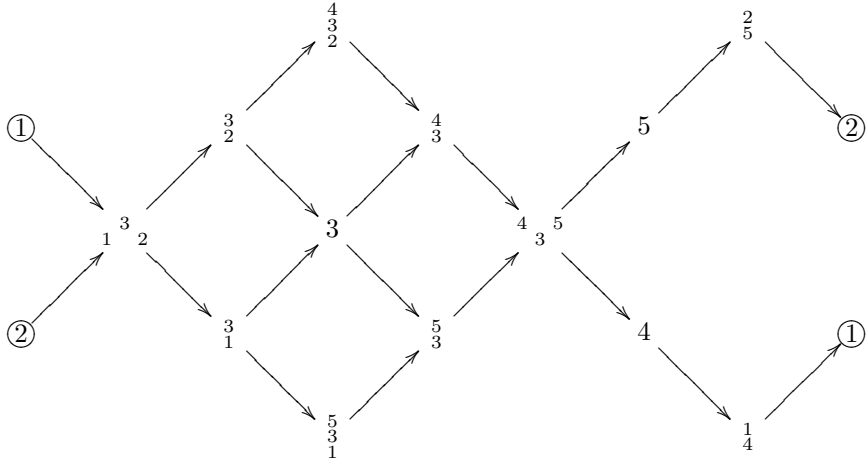
    \[
    \xymatrix@!@=12pt{
      &&&
      {\begin{smallmatrix}
          4\\3\\2
        \end{smallmatrix}}\ar[rd]
      &&&&
      {
        \begin{smallmatrix}
          2\\5
        \end{smallmatrix}}\ar[rd]
      \\
      *+[o][F-]{1}\ar[rd] &&
      {\begin{smallmatrix}
          3\\2
        \end{smallmatrix}}\ar[rd]\ar[ru]
      &&
      {\begin{smallmatrix}
          4\\3
        \end{smallmatrix}}\ar[rd]
      &&
      5\ar[ru] && *+[o][F-]{2}
      \\
      &
      {\begin{smallmatrix}
          &3&\\
          1&&2
        \end{smallmatrix}}\ar[rd]\ar[ru]
      &&
      3\ar[rd]\ar[ru]
      &&
      {\begin{smallmatrix}
          4&&5\\
          &3
        \end{smallmatrix}}\ar[rd]\ar[ru]
      \\
      *+[o][F-]{2}\ar[ru] &&
      {\begin{smallmatrix}
          3\\1
        \end{smallmatrix}}\ar[rd]\ar[ru]
      &&
      {\begin{smallmatrix}
          5\\3
        \end{smallmatrix}}\ar[ru]
      &&4\ar[rd]
      &&
      *+[o][F-]{1}\\
      &&&
      {\begin{smallmatrix}
          5\\3\\1
        \end{smallmatrix}}\ar[ru]
      &&&&
      {\begin{smallmatrix}
          1\\4
        \end{smallmatrix}}\ar[ru]
    }
    \]
    where the two encircled copies of $S_1=1$ are identified as are
    the two encircled copies of  $S_2=2$. It
    is easily seen that $\Gamma({\rm mod}\,B)$ is obtained from
    $\Gamma({\rm mod}\,\widetilde C)$ by deleting the $\widetilde
    C$-module $P_2=
    \begin{smallmatrix}
      2\\5
    \end{smallmatrix}$.
  \item Of course, one may have $\widetilde C$
    representation-infinite but $B$ representation-finite. Let $C$
    be given by the quiver
    
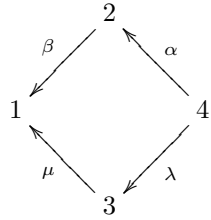
\[
    \xymatrix{
      & 2\ar[ld]_\beta \\
      1 && 4 \ar[lu]_\alpha \ar[ld]^\lambda \\
      & 3 \ar[lu]^\mu
    }
    \]
    bound by $\alpha\beta=0$, $\lambda\mu=0$. Its relation
    extension $\widetilde C$ is the cluster tilted algebra of type
    $\widetilde{\mathbb A}$ given by the quiver
    
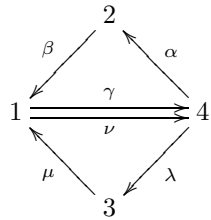
\[
    \xymatrix{
      & 2\ar[ld]_\beta \\
      1 \ar@<2pt>[rr]^\gamma \ar@<-2pt>[rr]_\nu && 4 \ar[lu]_\alpha \ar[ld]^\lambda \\
      & 3 \ar[lu]^\mu
    }
    \]
    and the potential $W=\alpha\beta\gamma+\lambda\mu\nu$. This is
    a representation-infinite algebra. However, if we let $E'$ be
    the direct summand of the $C-C$-bimodule corresponding to
    $W_1=\alpha\beta\gamma$, then $B=C\ltimes E'$ is given by the quiver
    
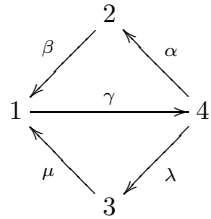
\[
    \xymatrix{
      & 2\ar[ld]_\beta \\
      1 \ar[rr]^\gamma  && 4 \ar[lu]_\alpha \ar[ld]^\lambda \\
      & 3 \ar[lu]^\mu
    }
    \]
    bound by $\alpha\beta=0$, $\beta\gamma=0$, $\gamma\alpha=0$,
    $\lambda\mu=0$. The algebra $B$ is representation-finite. Its Auslander-Reiten quiver is given by
    
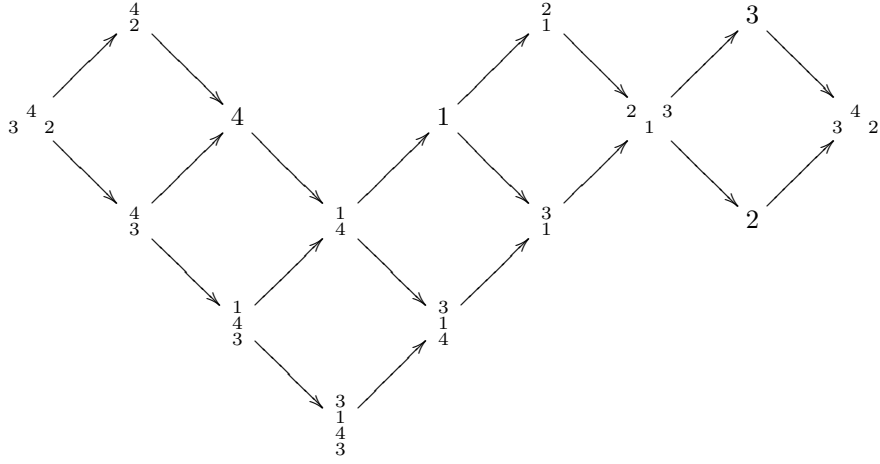
\[
    \xymatrix@!@=10pt{
      & {
        \begin{smallmatrix}
          4\\2
        \end{smallmatrix}
      }\ar[rd] &&&& {
        \begin{smallmatrix}
          2\\1
        \end{smallmatrix}
      }\ar[rd] && 3\ar[rd] \\
      {
        \begin{smallmatrix}
          &4\\3&&2
        \end{smallmatrix}
      }\ar[rd]\ar[ru]
      && 4\ar[rd] && 1\ar[rd]\ar[ru] && {
        \begin{smallmatrix}
          2&&3\\&1
        \end{smallmatrix}
      }\ar[rd]\ar[ru] && {
        \begin{smallmatrix}
          &4\\3&&2
        \end{smallmatrix}
      } \\
      & {
        \begin{smallmatrix}
          4\\3
        \end{smallmatrix}
      }\ar[rd]\ar[ru] && {
        \begin{smallmatrix}
          1\\4
        \end{smallmatrix}
      }\ar[rd]\ar[ru] && {
        \begin{smallmatrix}
          3\\ 1
        \end{smallmatrix}
      }\ar[ru] && 2\ar[ru]\\
      && {
        \begin{smallmatrix}
          1\\4\\3
        \end{smallmatrix}
      }\ar[ru]\ar[rd] && {
        \begin{smallmatrix}
          3\\1\\ 4
        \end{smallmatrix}
      }\ar[ru] \\
      &&&
      {
        \begin{smallmatrix}
          3\\1\\4\\ 3
        \end{smallmatrix}
      }\ar[ru]          
    }\]
    where the two copies of the module $P_4=
    \begin{smallmatrix}
      &4\\
      3&&2
    \end{smallmatrix}$ are identified.
  \end{enumerate}
\end{exple}

\section{Local slices}
\label{sec:local-slices-2}

\subsection{Preliminary facts}
\label{sec:preliminary-facts}

The notion of local slice was defined in \cite{ABS08} for the study of
cluster tilted algebras. We recall the definition.

\begin{defn}
  Let $A$ be an algebra. A full subquiver $\Sigma$ of $\Gamma({\rm
    mod}\,A)$ is called a \emph{local slice} if:
  \begin{enumerate}
  \item It is a presection, that is, if $L\to M$ is an irreducible
    morphism between indecomposables in ${\rm mod}\,A$, then
    \begin{enumerate}[(a)]
    \item $L\in \Sigma_o$ implies $M\in \Sigma_o$ or $\tau_AM\in
      \Sigma_o$,
    \item $M\in \Sigma_o$ implies $L\in \Sigma_o$ or $\tau_A^{-1}L\in \Sigma_o$.
    \end{enumerate}
  \item It is sectionally convex, that is, if $L=M_0\to M_1\to
    \cdots\to M_n=M$ is a sectional path in $\Gamma({\rm mod}\,A)$
    such that $L,M\in \Sigma_o$, then $M_i\in \Sigma_o$ for all $i$.
  \item $\vert \Sigma_o\vert= {\rm rk}(K_0(A))$.
  \end{enumerate}
  Here $\vert \Sigma_o\vert$ denotes the number of points of $\Sigma$.
\end{defn}

It is shown in \cite{ABS08} that, if $C$ is a tilted algebra, and
$\Sigma$ is a complete slice in $\Gamma({\rm mod}\,C)$, then $\Sigma$
embeds fully as a local slice in $\Gamma({\rm mod}\,\widetilde C)$,
where $\widetilde C$ denotes, as usual, the relation extension of $C$,
which is cluster tilted. However, local slices do not characterise
cluster tilted algebras, and it was asked in \cite{ABS08} to identify
the algebras which have local slices. Our objective in this section is
to prove that, if $A$ is an algebra such that there exist surjective
algebra morphisms $\widetilde C\twoheadrightarrow A\twoheadrightarrow
C$, then $A$ admits a local slice in its Auslander-Reiten quiver. For
this purpose, we need to recall the following well-known result of
Auslander and Reiten, see \cite[p. 187]{MR1476671}.

\begin{prop}
  \label{sec:preliminary-facts-1}
  Assume that there exists a surjective algebra homomorphism
  $A\twoheadrightarrow B$, and let $M$ be an indecomposable
  $B$-module. Then
  \begin{enumerate}[(a)]
  \item If $M$ is projective as an $A$-module, then $M$ is projective
    as a $B$-module. If $M$ is not projective as an $A$-module, then
    $\tau_BM$ is a submodule of $\tau_AM$.
  \item If $M$ is injective as an $A$-module, then $M$ is injective as
    a $B$-module. If $M$ is not injective as an $A$-module then
    $\tau_B^{-1}M$ is a quotient of $\tau_A^{-1}M$.
  \end{enumerate}
\end{prop}

\subsection{Modules on slices}
\label{sec:modules-slices}

We start with the following lemma.

\begin{lem}
  \label{sec:modules-slices-1}
  Let $C$ be a tilted algebra, $M$ a module on a complete slice
  $\Sigma$ in $\Gamma({\rm mod}\,C)$ and $\widetilde C$ the relation
  extension of $C$. Then:
  \begin{enumerate}[(a)]
  \item If $M$ is projective as a $C$-module, then it is projective as
    a $\widetilde C$-module. If $M$ is not projective as a $C$-module
    then $\tau_CM\simeq \tau_{\widetilde C}M$.
  \item If $M$ is injective as a $C$-module, then it is injective as a
    $\widetilde C$-module. If $M$ is not injective as a $C$-module,
    then $\tau_C^{-1}M\simeq \tau_{\widetilde C}^{-1}M$.
  \end{enumerate}
\end{lem}
\begin{proof}
  We only prove (a) because the proof of (b) is dual. Assume first
  that $M=eC$ is projective, with $e$ a primitive idempotent of
  $C$. Let, as usual, $E={\rm Ext}_C^2(DC,C)$. Because $\widetilde
  C=C\ltimes E$, it follows from \cite[Corollary 1.4]{AZ99}, that $M$
  is projective as a $\widetilde C$-module if and only if $eE=0$. Now
  $eE=e{\rm Ext}^2_C(DC,C)={\rm Ext}^2_C(DC,eC)\simeq {\rm
    Ext}^2_C(DC,M)=0$ because $M$, lying on a complete slice in ${\rm
    mod}\,C$, has injective dimension at most one.

  Assume now that $M$ is not projective. It follows from
  \cite[Theorem 2.1]{AZ99} that $\tau_CM\simeq \tau_{\widetilde C}M$
  if and only if $M\underset C \otimes E =0$ and ${\rm Hom}_C(E,\tau_C
  M)=0$. We proceed to prove these two equalities.

  Because $C$ is tilted and $\Sigma$ is a complete slice, the algebra
  $H={\rm End}_C\left(\oplus_{U\in \Sigma_o}U\right)$ is hereditary
  and there exists a tilting $H$-module $T$ such that $C={\rm
    End}_H(T)$. Because $M\in \Sigma_o$, there exists an injective
  $H$-module $I$ such that $M={\rm Hom}_H(T,I)$, see \cite[(VIII.3.5)
  and (VIII.5.6)]{MR2197389}. Denoting as before by $[-]$ the shift functor in
  the bounded derived category $\mathcal D^b({\rm mod}\,H)$ and by
  $\tau$ its Auslander-Reiten translation, it follows from
  \cite{MR2409188} that:
  \[
  \begin{array}{rcl}
    D(M\underset C \otimes E)
    & \simeq & {\rm Hom}_C(M,DE) \\
    & \simeq & {\rm Hom}_C( {\rm Hom}_H(T,I), D{\rm Hom}_{\mathcal
               D^b({\rm mod}\,H)}(T,\tau^{-1}T[1]) ) \\
    & \simeq & {\rm Hom}_C( {\rm Hom}_H(T,I), D{\rm Hom}_{\mathcal
               D^b({\rm mod}\,H)}(\tau T,T[1]) ) \\
    & \simeq & {\rm Hom}_C( {\rm Hom}_H(T,I), D{\rm Ext}^1_{\mathcal
               D^b({\rm mod}\,H)}(\tau T,T) ) \\
    & \simeq & {\rm Hom}_C( {\rm Hom}_H(T,I), {\rm Hom}_H(T,\tau_H^2T)
               ) \\
    & \simeq & {\rm Hom}_H(I,t(\tau_H^2T))
  \end{array}
  \]
  where $t(\tau_H^2T)={\rm Hom}_H(T,\tau_H^2T)\underset C\otimes T$ is
  the torsion submodule of $\tau_H^2T$ in the torsion pair $(\mathcal
  T(T_H), \mathcal F(T_H))$ induced by $T$ in ${\rm mod}\,H$, see
  \cite[(VI.3.9)]{MR2197389}. Now $\tau_H^2T$ is clearly not
  injective, therefore neither is its submodule
  $t(\tau_H^2T)$. Because $I$ is injective and $H$ is hereditary, we
  infer that ${\rm Hom}_H(T,t(\tau_H^2T))=0$. Therefore $M\underset C
  \otimes E=0$.

  The proof that ${\rm Hom}_C(E,\tau_CM)=0$ is sensibly different. We
  first claim that every indecomposable summand of $E_C$ is a proper
  successor of the complete slice $\Sigma$. Indeed, the
  Auslander-Reiten formula yields
  \[
  E={\rm Ext}^2_C(DC,C) \simeq {\rm Ext}^1_C(DC,\Omega^{-1}C) \simeq
  D\underline{\rm Hom}_C(\tau_C^{-1}\Omega^{-1}C,DC)\,.
  \]
  Now for any indecomposable summand $N$ of $\Omega^{-1}C$, there
  exists an indecomposable injective $C$-module $I_0$ such that ${\rm
    Hom}_C(I_0,N)\neq 0$. Because the slice $\Sigma$ is sincere in
  ${\rm mod}\,C$, there exist $L\in \Sigma_o$ and a nonzero morphism
  $L\to I_0$. Thus we have a path $L\to I_0\to N \to \star \to
  \tau_C^{-1}N$ in ${\rm mod}\,C$, so that $\tau_C^{-1}N$ is a proper
  successor of $\Sigma$ in ${\rm mod}\,C$. This proves that any
  indecomposable summand of $\tau_C^{-1}\Omega^{-1}C$ is a proper
  successor of $\Sigma$ in ${\rm mod}\,C$. On the other hand, no
  indecomposable projective $C$-module is a proper successor of
  $\Sigma$. Therefore
  \[
  \underline{\rm Hom}_C(\tau_C^{-1}\Omega^{-1}C,DC) \simeq
  {\rm Hom}_C(\tau_C^{-1}\Omega^{-1}C,DC)
  \]
  and so $E\simeq {\rm Hom}_C(\tau_C^{-1}\Omega^{-1}C,DC) \simeq
  \tau_C^{-1}\Omega^{-1}C$. This establishes our claim that every
  indecomposable summand of $E$ is a proper successor of $\Sigma$.

  Now $\tau_CM$ is a proper predecessor of $\Sigma$. Therefore ${\rm
    Hom}_C(E,\tau_CM)=0$. This completes the proof. 
\end{proof}

\begin{prop}
  \label{sec:modules-slices-2}
  Let $C$ be a tilted algebra, $M$ a module in a complete slice
  $\Sigma$ in $\Gamma({\rm mod}\,C)$, $\widetilde C$ the relation
  extension algebra and $A$ an algebra such that there exist
  surjective algebra morphisms $\widetilde C\twoheadrightarrow
  A\twoheadrightarrow C$. Then:
  \begin{enumerate}[(a)]
  \item If $M$ is projective as a $C$-module, then it is projective as
    an $A$-module. If $M$ is not projective as a $C$-module, then
    $\tau_CM\simeq \tau_AM$.
  \item If $M$ is injective as a $C$-module, then it is injective as
    an $A$-module. If $M$ is not injective as a $C$-module, then
    $\tau_C^{-1}M\simeq \tau^{-1}_AM$.
  \end{enumerate}
\end{prop}
\begin{proof}
  This follows from lemma~\ref{sec:modules-slices-1} and proposition~\ref{sec:preliminary-facts-1}.
\end{proof}

\begin{cor}
  \label{sec:modules-slices-3}
  Let $C$ be a tilted algebra, $\Sigma$ a complete slice in
  $\Gamma({\rm mod}\,C)$, $\widetilde C$ the relation extension of $C$
  and $A$ an algebra such that there exist surjective algebra
  morphisms $\widetilde C\twoheadrightarrow A\twoheadrightarrow
  C$. Let $L\to M$ be an irreducible morphism between indecomposables
  in ${\rm mod}\,A$. If either $L$ or $M$ lies in $\Sigma$, then the
  other is a $C$-module.
\end{cor}
\begin{proof}
  We may, by duality, assume that $L\in \Sigma_o$. Suppose first that
  $L$ is an injective $C$-module. Because of proposition
  \ref{sec:modules-slices-3}, it is injective as an $A$-module. In
  particular, ${\rm soc}_CL={\rm soc}_AL$ and so the canonical
  projection $L\twoheadrightarrow L/{\rm soc}_CL$ is a minimal left
  almost split morphism in ${\rm mod}\,A$. Therefore $M$ is an
  indecomposable direct
  summand of $L/{\rm soc}_CL$ and in particular is a $C$-module.

  Suppose that $L$ is not injective as a $C$-module. Because of
  proposition~\ref{sec:modules-slices-3}, we have $\tau_C^{-1}L\simeq
  \tau_A^{-1}L$. It then follows from \cite[Theorem 2.1]{AZ99} that
  the almost split sequence $0\to L\to X\to \tau_C^{-1}L\to 0$  in
  ${\rm mod}\,C$ remains almost split in ${\rm mod}\,A$. Therefore $M$
  is an indecomposable direct summand of $X$, so it is a
  $C$-module. This completes the proof.
\end{proof}

\subsection{The existence of local slices}
\label{sec:exist-local-slic}

We are now able to prove the main result of this section.

\begin{thm}
  \label{sec:exist-local-slic-2}
    Let $C$ be a tilted algebra and $A$ be an algebra such that there
  exist surjective algebra morphisms $\widetilde C\twoheadrightarrow
  A\twoheadrightarrow C$, then any complete slice in $\Gamma({\rm
    mod}\,C)$ embeds as a local slice in $\Gamma({\rm mod}\,A)$. In
  particular, partial relation extensions admit local slices.
\end{thm}
\begin{proof}
  Because clearly $\vert \Sigma_o \vert ={\rm rk}(K_0(C))={\rm
    rk}(K_0(A))$, it suffices to prove the first two properties in the
  definition of local
  slices.

  We first show that $\Sigma$ is a presection in $\Gamma({\rm
    mod}\,A)$. Let $f\colon L\to M$ be an irreducible morphism between
  indecomposables in ${\rm mod}\,A$. Assume $L\in \Sigma$. Because of
  corollary~\ref{sec:modules-slices-3}, $M$ is a $C$-module. Therefore
  $f$ remains an irreducible morphism in ${\rm mod}\,C$. Because the
  complete slice $\Sigma$ is a presection in $\Gamma({\rm mod}\,C)$,
  we have $M\in \Sigma_o$ or $\tau_CM\in \Sigma_o$. In the latter
  case, the observation that $\tau_CM\simeq \tau_AM$ completes the
  proof.

  One shows in exactly the same way that, if $M\in \Sigma_o$, then $L\in
  \Sigma_o$ or $\tau_A^{-1}L\in \Sigma_o$.

  There remains to prove sectional convexity. Let
  \[
  M_0\xrightarrow{f_1}M_1\xrightarrow{f_2}M_2\to\cdots
  \xrightarrow{f_t}M_t
  \]
  be a sectional path in $\Gamma({\rm mod}\,A)$, with $M_0,M_t\in
  \Sigma$. We may assume without loss of generality  that
  $M_1\not\in \Sigma_o$. Because of
  corollary~\ref{sec:modules-slices-3}, $M_1$ is a $C$-module. Now,
  observe that the morphism $f_t\cdots f_2\colon M_1\to M_t$ is
  nonzero in ${\rm mod}\,A$, because it is the composition of a
  sectional path. Therefore it is also nonzero in ${\rm
    mod}\,C$. Because $f_1\colon M_0\to M_1$ is also nonzero in ${\rm
    mod}\,C$,  the
  convexity of $\Sigma$ in ${\rm mod}\,C$ and the path $M_0\to M_1\to
  M_t$ yield $M_1\in \Sigma_o$, a contradiction which completes the proof.
\end{proof}

In particular, our result applies to partial relation extensions.

\begin{cor}
  \label{sec:exist-local-slic-1}
  Let $C$ be a tilted algebra and $B$  a partial relation
  extension. Then any complete slice in $\Gamma({\rm mod}\,C)$ embeds
  as a local slice in $\Gamma({\rm mod}\,B)$.\hfill$\square$
\end{cor}

The reader may notice that the example in \cite{ABS08} of a local
slice is an example of a local slice in a partial relation
extension. We give an example of an algebra which has a local slice
but is not a partial relation extension.

\begin{exple}
  \label{sec:exist-local-slic-3}
  Let $A$ be given by the quiver
  
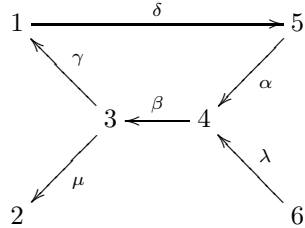
\[
  \xymatrix{
    1 \ar[rrr]^\delta & & & 5 \ar[ld]^\alpha \\
    & 3 \ar[lu]_\gamma \ar[ld]^\mu & 4 \ar[l]_\beta & \\
    2 & & & 6 \ar[lu]_\lambda
  }
  \]
  bound by $\lambda\beta\mu=0$, $\alpha\beta \gamma=0$,
  $\gamma\delta=0$, $\delta\alpha=0$. Then $\Gamma({\rm mod}\,A)$ is
  given by
    
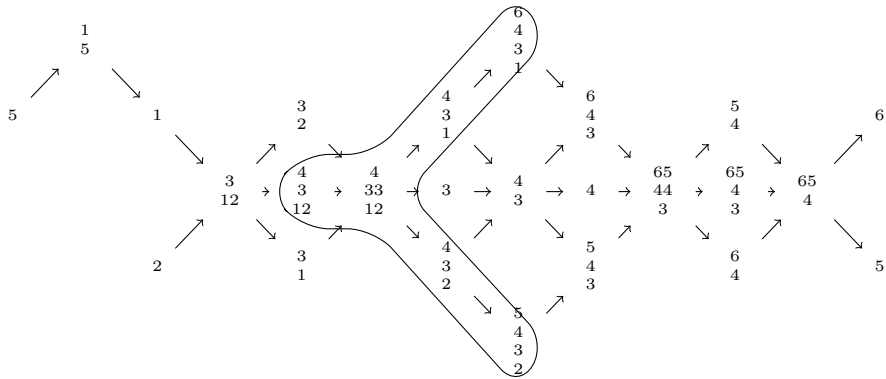
\[
  \begin{tikzpicture}[font=\tiny,xscale=1.1,yscale=1.1]

    \node (5) at (0,0) {$
      \begin{array}{c}
        5
      \end{array}$};
    
    \node (15) at (1,1) {$
      \begin{array}{c}
        1\\5
      \end{array}$};

    \node (1) at (2,0) {$
      \begin{array}{c}
        1
      \end{array}$};

    \node (312) at (3,-1) {$
      \begin{array}{c}
        3\\12
      \end{array}$};

    \node (2) at (2,-2) {$
      \begin{array}{c}
        2
      \end{array}$};

    \node (32) at (4,0) {$
      \begin{array}{c}
        3\\2
      \end{array}$};

    \node (4312) at (4,-1) {$
      \begin{array}{c}
        4\\3\\12
      \end{array}$};

    \node (31) at (4,-2) {$
      \begin{array}{c}
        3\\1
      \end{array}$};

    \node (43312) at (5,-1) {$
      \begin{array}{c}
        4\\33\\12
      \end{array}$};

    \node (431) at (6,0) {$
      \begin{array}{c}
        4\\3\\1
      \end{array}$};

    \node (3) at (6,-1) {$
      \begin{array}{c}
        3
      \end{array}$};

    \node (432) at (6,-2) {$
      \begin{array}{c}
        4\\3\\2
      \end{array}$};

    \node (6431) at (7,1) {$
      \begin{array}{c}
        6\\4\\3\\1
      \end{array}$};

    \node (43) at (7,-1) {$
      \begin{array}{c}
        4\\3
      \end{array}$};

    \node (5432) at (7,-3) {$
      \begin{array}{c}
        5\\4\\3\\2
      \end{array}$};

    \node (643) at (8,0) {$
      \begin{array}{c}
        6\\4\\3
      \end{array}$};

    \node (4) at (8,-1) {$
      \begin{array}{c}
        4
      \end{array}$};

    \node (543) at (8,-2) {$
      \begin{array}{c}
        5\\4\\3
      \end{array}$};

    \node (65443) at (9,-1) {$
      \begin{array}{c}
        65\\44\\3
      \end{array}$};

    \node (54) at (10,0) {$
      \begin{array}{c}
        5\\4
      \end{array}$};

    \node (6543) at (10,-1) {$
      \begin{array}{c}
        65\\4\\3
      \end{array}$};

    \node (64) at (10,-2) {$
      \begin{array}{c}
        6\\4
      \end{array}$};

    \node (654) at (11,-1) {$
      \begin{array}{c}
        65\\4
      \end{array}$};

    \node (6) at (12,0) {$
      \begin{array}{c}
        6
      \end{array}$};

    \node (5bis) at (12,-2) {$
      \begin{array}{c}
        5
      \end{array}$};

    \draw[->] (5) -- (15);

    \draw[->] (15) -- (1);

    \draw[->] (1) -- (312);

    \draw[->] (2) -- (312);

    \draw[->] (312) -- (32);

    \draw[->] (312) -- (4312);

    \draw[->] (312) -- (31);

    \draw[->] (32) -- (43312);

    \draw[->] (4312) -- (43312);

    \draw[->] (31) -- (43312);

    \draw[->] (43312) -- (431);

    \draw[->] (43312) -- (3);

    \draw[->] (43312) -- (432);

    \draw[->] (431) -- (6431);

    \draw[->] (431) -- (43);

    \draw[->] (3) -- (43);

    \draw[->] (432) -- (43);

    \draw[->] (432) -- (5432);

    \draw[->] (5432) -- (543);

    \draw[->] (43) -- (4);

    \draw[->] (43) -- (543);

    \draw[->] (43) -- (643);

    \draw[->] (643) -- (65443);

    \draw[->] (4) -- (65443);

    \draw[->] (543) -- (65443);

    \draw[->] (65443) -- (54);

    \draw[->] (65443) -- (6543);

    \draw[->] (64) -- (654);

    \draw[->] (654) -- (6);

    \draw[->] (654) -- (5bis);

    \draw[->] (6431) -- (643);

    \draw[->] (65443) -- (64);

    \draw[->] (54) -- (654);

    \draw[->] (6543) -- (654);

    \draw[rounded corners=10pt]
    (5432.south)
    -- (43312.south) -- (4312.south) -- (4312.west) --
    (4312.north) -- (43312.north)
    -- (6431.north) -- (6431.east) -- (43312.east) -- 
    (5432.east) --
    cycle;
\end{tikzpicture}
  \]
  where the two copies of $5$ are identified. We have illustrated
  a local slice which arises from the embedding of $\Gamma({\rm
    mod}\,C)$ in $\Gamma({\rm mod}\,A)$, where $C$ is the algebra
  obtained from 
  $A$ by deleting the arrow $\delta$ (that is, $C=A/\langle
  \delta\rangle$). Notice that $C$ is a tilted algebra of type
  $\mathbb E_6$. Notice finally that $A$ is not a partial relation
  extension as shown by direct inspection: for instance, if $A$ were a
  partial relation extension of $C$, then the defining relations of
  $A$ involving the arrow $\gamma$ would be $\beta\gamma\delta=0$ and
  $\delta\alpha\beta=0$ (instead of $\gamma\delta=0$ and
  $\delta\alpha=0$) because the Keller potential has a unique oriented
  cycle, namely, $\alpha\beta\gamma\delta$, containing $\alpha$ or $\gamma$.
\end{exple}

\begin{center}
  \textsc{Acknowledgements}
\end{center}
Part of this work was done during visits of the fourth named author to
his coauthors at Université de Sherbrooke. He thanks them for their
warm hospitality during his stays.

The first named author gratefully acknowledges partial support from
the NSERC of Canada, FRQNT of Québec and the Université de Sherbrooke.

The fourth named author acknowledges financial support from CRM (UMI
CNRS 3457).

The fifth named author acknowledges financial support from the NSERC
of Canada.

%\bibliographystyle{plain}
%\bibliography{biblio}

\end{document}